\documentclass[A4,twoside,10pt,reqno]{article}
\usepackage{amsfonts, amssymb}
\usepackage[figuresright]{rotating}
\usepackage{epstopdf, float, graphics, graphicx, subcaption, tikz, url}
\usepackage{amsthm, latexsym, amscd, doc, epsfig, eucal, mathrsfs, enumerate, fancyhdr}
\usepackage[tbtags]{amsmath}
\usepackage[frame,cmtip,arrow,matrix,line,graph,curve]{xy}
\usetikzlibrary{backgrounds}

\newcommand\shorttitle{Some spaces of polynomial knots}
\newcommand\authors{Hitesh Raundal and Rama Mishra}
\title{Some spaces of polynomial knots}
\author{Hitesh Raundal{$^\ast$} and Rama Mishra{$^\dagger$}}
\date{{\it\small Department of Mathematics, Indian Institute of Science Education and}\\ {\it\small Research, Dr. Homi Bhabha Road, Pashan, Pune - 411008,  India.}\\{\it\footnotesize{$^\ast$}hiteshrndl@gmail.com\hskip2mm and\hskip2mm {$^\dagger$}ramamishra64@gmail.com}}

\numberwithin{equation}{section}

\newtheorem{theorem}{Theorem}[section]  
\newtheorem{lemma}[theorem]{Lemma}
\newtheorem{proposition}[theorem]{Proposition}
\newtheorem{definition}[theorem]{Definition}

\newtheorem{remark}[theorem]{Remark}
\newtheorem{corollary}[theorem]{Corollary}

\newcommand{\mcup}{\mathbin{\scalebox{0.9}{\ensuremath{\bigcup}}}}
\newcommand\xbar[1]{\hbox{\vbox{\hrule height0.4pt\kern0.25ex\hbox{\kern-0.17em    %
\ensuremath{#1}\kern-0.05em}}}}  

\newcommand{\hp}{\hskip0.2mm}
\newcommand{\hf}{\hskip0.4mm}
\newcommand{\hs}{\hskip0.6mm}
\newcommand{\ho}{\hskip1mm}
\newcommand{\hw}{\hskip2mm}
\newcommand{\vo}{\vskip1mm}
\newcommand{\vw}{\vskip2mm}

\newcommand{\spm}{\mathcal{A}}
\newcommand{\spk}{\mathcal{P}}
\newcommand{\kd}{\mathcal{K}_d}
\newcommand{\vd}{\mathcal{V}_d}
\newcommand{\ad}{\mathcal{A}_d}
\newcommand{\bd}{\mathcal{B}_d}
\newcommand{\cd}{\mathcal{C}_d}
\newcommand{\od}{\mathcal{O}_d}
\newcommand{\pd}{\mathcal{P}_d}
\newcommand{\qd}{\mathcal{Q}_d}
\newcommand{\ro}{\mathbb{R}}
\newcommand{\rw}{\mathbb{R}^2}
\newcommand{\rt}{\mathbb{R}^3}
\newcommand{\rtd}{\mathbb{R}^{3d}}
\newcommand{\so}{S^1}
\newcommand{\sw}{S^2}
\newcommand{\st}{S^3}

\newcommand{\zo}{I}
\newcommand{\ozo}{(0,1)}

\newcommand{\cmoo}{\{\hp-1,1\hp\}}

\newcommand{\fgh}{(f, g, h)}

\newcommand{\abc}{(a_0,a_1,\ldots,a_{d-2},b_0,b_1,\ldots,b_{d-1},c_0,c_1,\ldots,c_d)}

\newcommand{\ba}{\mathbf{a}}
\newcommand{\bb}{\mathbf{b}}
\newcommand{\bc}{\mathbf{c}}
\newcommand{\babc}{(\ba,\bb,\bc)}
\newcommand{\abcst}{(\ba,\bb,\bc,s,t)} 

\newcommand{\fght}{\big(\hp f(t),\hp g(t),\hp h(t)\hp\big)}
\newcommand{\fgmht}{\big(\hp f(t),\hp g(t),\hp -h(t)\hp\big)}
\newcommand{\uvwt}{\big(\hp u(t),\hp v(t),\hp w(t)\hp\big)}
\newcommand{\xyzt}{\big(\hp x(t),\hp y(t),\hp z(t)\hp\big)}

\newcommand{\at}{a_0+a_1t+\cdots+a_{d-2}t^{d-2}}
\newcommand{\bt}{b_0+b_1t+\cdots+b_{d-1}t^{d-1}} 
\newcommand{\ct}{c_0+c_1t+\cdots+c_dt^d} 
\newcommand{\abct}{\left(\at,\,\bt,\,\ct\right)}

\newcommand{\past}{a_1+a_2\left(s+t\right)+\cdots+a_{d-2}\left(s^{d-3}+s^{d-4}t+\cdots+t^{d-3}\right)}
\newcommand{\pbst}{b_1+b_2\left(s+t\right)+\cdots+b_{d-1}\left(s^{d-2}+s^{d-3}t+\cdots+t^{d-2}\right)}
\newcommand{\pcst}{c_1+c_2\left(s+t\right)+\cdots+c_d\left(s^{d-1}+s^{d-2}t+\cdots+t^{d-1}\right)}\newcommand{\phiabc}{\phi_{\ba\bb\bc}}

\begin{document}
	
\maketitle

\begin{abstract}
{\it In this paper we study the topology of three different kinds of spaces associated to polynomial knots of degree at most $d$, for $d\geq2$. We denote these spaces by $\od$, $\pd$ and $\qd$. For $d\geq3$, we show that the spaces $\od$ and $\pd$ are path connected and the space $\od$ has the same homotopy type as $\sw$. Considering the space $\spk=\mcup_{d\geq2}\od$ of all polynomial knots with the inductive limit topology, we prove that it too has the same homotopy type as $\sw$. We also show that if two polynomial knots are path equivalent in $\qd$, then they are topologically equivalent. Furthermore, the number of path components in $\qd$ are in multiples of eight.}
\end{abstract}

\section{Introduction}\label{sec1}

Parameterizing knots has been useful in estimating some important knot invariants such as bridge index \cite{hs}, superbridge index \cite{nk} and geometric degree \cite{jw}. Some of the interesting parameterizations are Fourier knots \cite{lk}, polygonal knots \cite{jac} and polynomial knots \cite{vav2}. The first two of them  provide classical knots, whereas the polynomial parametrization gives long knots. In each parametrization there is a positive integer $d$ associated to it. For Fourier knots and polynomial knots it is its {\it degree} and for polygonal knots it is its {\it edge number}. For each parametrization, one can study the space of all parametrized knots for a fixed $d$. Once we fix $d$, there will be only finitely many knots that can be parametrized with this $d$. However, in the space of parameterizations, one could study the topology and try to see if two parametrized knots in this space belong to the same path components or not. In \cite{jac}, Calvo studied the spaces of polygonal knots. We aim to study the topology of some spaces of polynomial knots. In \cite{vav2}, Vassiliev defined the space $\vd$ to be the interior of the set of all smooth embeddings in the space $\mathcal{W}_d$ of all polynomial maps of the type $t\mapsto\left(t^d+a_{d-1}t^{d-1}+\cdots+a_1t,\hf t^d+b_{d-1}t^{d-1}+\cdots+b_1t,\hf t^d+c_{d-1}t^{d-1}+\cdots+c_1t\right)$. He pointed out that the space of long knots can be approximated by the spaces $\vd$ for $d\geq1$ (see \cite{vav1}). Also, it was noted that if two polynomial knots belong to the same path component of $\vd$ then they represent the same knot-type. It is felt that the converse of this may not be true. However, no counter example is known. Regarding the topology of the space $\vd$, he showed that the space $\mathcal{V}_3$ is contractible and the space $\mathcal{V}_4$ is homology equivalent to $\so$.\vskip0.1mm
Later, Durfee and O'Shea \cite{do} introduced a different space $\kd$ which is the space of all knots $t\mapsto\left(a_0+a_1t+\cdots+a_dt^d,\,b_0+b_1t+\cdots+b_dt^d,\,c_0+c_1t+\cdots+c_dt^d\right)$ such that $\left|a_d\right|+\left|b_d\right|+\left|c_d\right|\neq0$. If two polynomial knots are path equivalent in $\kd$, then they are topologically equivalent. Thus, the space $\kd$ will have at least as many path components as knot-types that can be represented in it. In \cite{sky}, a group of undergraduate students proved that $\mathcal{K}_5$ has at least $3$ path components corresponding to the unknot, the right hand trefoil and the left hand trefoil respectively. Beyond this the topology of these spaces is not understood.  Also, it is not clear whether two topologically equivalent knots in this space necessarily belong to the same path component or not. Composing a polynomial knot with a simple polynomial automorphism of $\rt$ can reduce the degree of two of the components and the resulting polynomial knot will be topologically equivalent to the earlier one. Keeping this view in mind, we note that any polynomial knot given by $t\mapsto\left(a_0+a_1t+\cdots+a_dt^d,\,b_0+b_1t+\cdots+b_dt^d,\,c_0+c_1t+\cdots+c_dt^d\right)$, for $d\geq2$, is topologically equivalent to a polynomial knot $t\mapsto\fght$ with $\deg(f)\leq d-2,$ $\deg(g)\leq d-1$ and $\deg(h)\leq d$. This motivated us to study the topology of three interesting spaces namely: (1) the space $\od$ of all polynomial knots $t\mapsto\fght$ with $\deg(f)\leq d-2,\hf\deg(g)\leq d-1$ and $\deg(h)\leq d$, (2) the space $\pd$ of all polynomial knots $t\mapsto\uvwt$ with $\deg(u)<\deg(v)<deg(w)\leq d$, and (3) the space $\qd$ of all polynomial knots $t\mapsto\xyzt$ with $\deg(x)=d-2,\hf\deg(y)=d-1$ and $\deg(z)=d$. Using the theory of real semialgebraic sets (see \cite{mc} and \cite{br}), we ensure that these spaces must have finitely many path components. We show that the space $\od$, for $d\geq 3$, has the same homotopy type as $\sw$. We also prove that the space $\pd$, for $d\geq 3$, is path connected, whereas the space $\qd$ is not path connected. For the space $\qd$, for $d\geq2$, we show that if two polynomial knots lie in the same path component then they are topologically equivalent. We provide a counter example that the converse of this is not true. Furthermore, we show that the space $\spk$ of all polynomial knots, with the inductive limit topology coming from the stratification $\spk=\mcup_{d\geq2}\od$, also has the same homotopy type as $\sw$.\vskip0.1mm This paper is organized as follows: Section 2 is about definitions and known results. We divide it in three subsections. In 2.1 and 2.2, we provide the basic terminologies and some known results related to knots and in particular polynomial knots which will be required in this paper. In 2.3, we discuss real semialgebraic sets and mention some important results from real semialgebraic geometry. In Section 3, we introduce our spaces $\od,\hp\pd$ and $\qd$ for $d\geq2$ and check their basic topological properties. At the end of Section 3, we show that these spaces are homeomorphic to some semialgebraic subsets of $\ro^{3d}$ and hence they have only finitely many path components. In Section 4, we prove our main results which are the following:\\[4pt]
{\bf Theorem \ref{th4}}: {\it The space $\pd$, for $d\geq3$, is path connected.}\\[4pt]
{\bf Theorem \ref{th13}}: {\it If two polynomial knots are path equivalent in $\qd$, then they are topologically equivalent.}\\[4pt]
{\bf Theorem \ref{th15}}: {\it Every polynomial knot is isotopic to some trivial polynomial knot by a smooth isotopy of polynomial knots.}\\[4pt]
{\bf Corollary \ref{th22}}: {\it Every polynomial knot is connected to a trivial polynomial knot by a smooth path in the space $\spk$ of all polynomial knots.}\\[4pt]
{\bf Theorem \ref{th3}}: {\it The space $\od$, for $d\geq3$, has the same homotopy type as $\sw$.}\\[4pt]
{\bf Corollary \ref{th20}}: {\it The space $\spk$ of all polynomial knots has the same homotopy type as $\sw$.}\\[4pt]
{\bf Theorem \ref{th5.13}}: {\it The path components of the space $\mathcal{Q}_4$ are contractible.}\\[4pt]
{\bf Proposition \ref{th5.15}}: {\it The space $\mathcal{Q}_5$ has exactly eight path components corresponding to the trefoil knot and its mirror image. Moreover, all these path components are contractible.}

\section{Definitions and Known Results}\label{sec2}

\subsection{Knots and their isotopies}\label{sec2.1}

\begin{definition}
A tame knot is a continuous embedding $\kappa:\so\to\st$ which can be extended to a continuous embedding $\bar{\kappa}:\so\times B^2\to\st$ of a tubular neighborhood of $\so$ in $\rt$, where $B^2$ is the open unit disc in $\rw$.
\end{definition}

It is known that every piecewise $C^1$ embedding of $S^1$ in $S^3$ is a tame knot (see \cite[Lemma~2]{do} and \cite[Appendix~I]{cf}). In particular, all smooth embeddings (that is, $C^\infty$ embeddings) of $S^1$ in $S^3$ are tame knots.

\begin{definition}
A long knot is a proper smooth embedding $\phi:\ro\to\rt$ such that the map $t\mapsto\left\|\phi(t)\right\|$ is strictly monotone outside some closed interval $[a,b]\subset\ro$ and $\left\|\phi(t)\right\|\to\infty$ as $\left| t\right| \to\infty$.
\end{definition}

\begin{definition}\label{def1}
A string isotopy of tame knots is a continuous map $F:\zo\times\so\to\st$ such that for each $s\in\zo$, the map $F_s=F(s,\hf\cdot\hf)$ is a tame knot.
\end{definition}

\begin{definition}\label{def2}
Two tame knots $\phi,\psi:\so\to\st$ are string isotopic if there is a string isotopy $F:\zo\times\so\to\st$ of tame knots such that $F_0=\phi$ and $F_1=\psi$.
\end{definition}

\begin{definition}\label{def3}
A homeotopy (respectively, diffeotopy) of the space $\st$ is a continuous map $H:\zo\times\st\to\st$ such that:
\begin{enumerate}[(1)]
\item for each $s\in\zo$, the map $H_s=H(s,\hf\cdot\hf)$ is a homeomorphism (respectively, a diffeomorphism) of $\st$, and
\item the map $H_0=H(0,\hf\cdot\hf)$ is the identity map of $\st$.
\end{enumerate}
\end{definition}

Note that self-homeomorphisms and self-diffeomorphisms are automorphisms in continuous and smooth categories respectively, and homeotopies and diffeotopies are ambient isotopies in the respective categories.

\begin{definition}\label{def4}
Two tame knots $\phi,\psi:\so\to\st$ are said to be ambient isotopic if there is a homeotopy $H:\zo\times\st\to\st$ of the ambient space $\st$ such that $\psi=H_1\circ\phi$.
\end{definition}

Similar terms as in Definitions \ref{def1} to \ref{def4} can be defined for the category of smooth knots $\so\hookrightarrow\st$ and also for the category of long knots $\ro\hookrightarrow\rt$. In these categories, the automorphisms of the spaces are self-diffeomorphisms instead of self-homeomorphisms. Also, for these categories, the string isotopies are isotopies of knots in the same category, and the ambient isotopies are diffeotopies instead of homeotopies.\vskip0.1mm
Every long knot $\phi:\ro\to\rt$ can be extended to a unique embedding $\tilde{\phi}:\so\to\st$ by the inverse of the stereographic projection from the north pole of $\st$. On the other hand, every ambient isotopy class of tame knots contains a smooth knot $\psi:\so\to\st$ which fixes the north poles and has nonzero derivative at the north pole. The restriction $\hat{\psi}:\ro\to\rt$ of $\psi$ is a long knot. Note that $\hat{\tilde{\phi}}\simeq\phi$ and $\tilde{\hat{\psi}}\simeq\psi$, where \textquoteleft$\simeq$\textquoteright\, denotes that one knot is ambient isotopic to the other. Thus, there is bijection between the ambient isotopy classes of tame knots and the ambient isotopy classes of long knots.

\subsection{Polynomial knots}\label{sec2.2}

\begin{definition} 
A polynomial map in $\rt$ is a map $\phi:\ro\to\rt$ whose component functions are real polynomials.
\end{definition}

\begin{definition}
A polynomial knot in $\rt$ is a polynomial map $\phi:\ro\to\rt$ which is a smooth embedding; that is, it does not have a multiple or a critical point. 
\end{definition}

\begin{definition} 
A polynomial knot $\phi$ is said to represent a knot-type $[\kappa]$ (that is, the ambient isotopy class of $\kappa$) if the extended knot $\tilde{\phi}:\so\to\st$ is ambient isotopic to $\kappa$. In this case, the knot $\phi$ is a polynomial representation of the knot-type $[\kappa]$. 
\end{definition}

It was proved that every long knot is ambient isotopic to some polynomial knot (see \cite{ars}). This implies that every knot-type can be represented by some polynomial knot. In other words, every knot-type has a polynomial representation.\vskip0.1mm

If a polynomial knot $\phi$ given by $t\mapsto\fght$ represents a knot $\kappa$, then the knot $t\mapsto\fgmht$ represents a mirror image $\kappa^*$ of $\kappa$. In general, for an affine transformation (a composition of an invertible linear transformation and a translation) $S:\rt\to\rt$, the polynomial knot $S\circ\phi$ represents either $\kappa$ or $\kappa^*$ depending on whether $S$ is orientation preserving or not. 

\begin{definition}
A polynomial map $t\mapsto\fght$ is said to have degree sequence $(d_1,d_2,d_3)$ if $\deg(f)=d_1$, $\deg(g)=d_2$ and $\deg(h)=d_3$.
\end{definition}  

\begin{definition}
The degree of a polynomial map $\phi:\ro\to\rt$ is the maximum of the degrees of its component polynomials.
\end{definition}

\begin{definition} 
A knot-type $[\kappa]$ is said to have polynomial degree $d$, if there exists a polynomial knot $\phi$ of degree $d$ which represents it and there is no polynomial knot of degree less than $d$ representing it. In this case, the polynomial knot $\phi$ is a minimal polynomial representation of the knot-type $[\kappa]$.
\end{definition}

For example, the polynomial degree of a trivial knot is $1$ and $t\mapsto (0,0,t)$ is its one of the minimal polynomial representation.

 \begin{definition}
 A polynomial isotopy is a continuous map $H:\zo\times\ro\to\rt$ such that $H_s=H(s,\hf\cdot\hf)$ is a polynomial knot for each $s\in\zo$.
 \end{definition}
 
 \begin{definition}
 Two polynomial knots $\phi$ and $\psi$ are polynomially isotopic if there is a polynomial isotopy $H:\zo\times\ro\to\rt$ such that $H_0=\phi$ and $H_1=\psi$.
 \end{definition}

\begin{definition}
A polynomial automorphism is a bijective map $T:\mathbb{R}^n\rightarrow\mathbb{R}^n$ whose component functions are real polynomials in $n$ variables and it is such that its inverse is also of the same kind. 
\end{definition}

Note that a map $S:\ro\to\ro$\, is a polynomial automorphism if and only if it is a linear polynomial. Also, it is easy to see that an affine transformation (a composition of an invertible linear transformation and a translation) of $\mathbb{R}^n$ is a polynomial automorphism.\vskip0.1mm
Regarding the polynomial maps and the polynomial automorphisms of $\ro$ and $\rt$, the following remarks can easily be verified:
\begin{remark}\label{rm1} 
Let $\phi:\ro\rightarrow\rt$ be a polynomial map, and let $S:\ro\rightarrow\ro$ and $T:\rt\rightarrow\rt$ be polynomial automorphisms. Then the map $\phi$ is a smooth embedding if and only if the composition $T\circ\phi\circ S$ is a smooth embedding.
\end{remark}

\begin{remark}\label{rm2}
Let $\alpha$ and $\gamma $ be real numbers and let $\alpha\neq 0$. Then for a polynomial knot  $\phi:\ro\rightarrow\rt$, the map $\psi:\ro\rightarrow\rt$ given by $t\mapsto\phi(\alpha t+\gamma)$ is also a polynomial knot.
\end{remark} 

\begin{remark}\label{rm3}  
Let $\phi:\ro\rightarrow\rt$ be a polynomial knot. Then for an affine transformation $T:\rt\rightarrow\rt$, the map $T\circ\phi$ is also a polynomial knot.
\end{remark} 

Let $\phi:\ro\to\rt$ be a polynomial knot. Then for an orientation preserving polynomial automorphisms $S:\ro\to\ro$ and $T:\rt\to\rt$ (that is, $S$ is a linear polynomial with positive leading coefficient and $\det\big(\frac{\partial T_i}{\partial x_j}(x)\big)>0$ for all $x\in\rt$), the polynomial knot $T\circ\phi\circ S$ is topologically equivalent to $\phi$. In particular, the same is true if $S$ and $T$ are orientation preserving affine transformations.\vskip0.1mm 
We provide the following proposition (see \cite[Proposition 3]{do}) which we will be using in this paper.

\begin{proposition}\label{th16}
Let $\{\alpha_s\mid s\in\zo\}$ be a family of polynomial knots depending continuously on parameter $s\in\zo$. If the degree of $\alpha_s$ is independent of $s$, then there exists $r_0\gg0$ such that for any $s\in\zo$ and any $\left|t\right|\geq r_0$, the vector $\alpha_s'(t)\in\rt$ intersects transversely to the sphere of radius $\left|\alpha_s(t)\right|$ about the origin. Moreover, for any $s\in\zo$, the angle between the vectors $\alpha_s(t)$ and $\alpha_s'(t)$ approaches $0$ as $t\to+\infty$ and it approaches $\pi$ as $t\to-\infty$. 
\end{proposition}

\begin{proof}
For $s\in\zo$, let $\alpha_s$ be given by
\begin{equation*}
\alpha_s(t)=\ba_d(s)\hp t^d+\ba_{d-1}(s)\hp t^{d-1}+\cdots+\ba_1(s)\hp t+\ba_0(s)
\end{equation*}
for $t\in\ro$, where $\ba_i$'s are continuous functions from $\zo$ to $\rt$ and $\ba_d(s)\neq0$ for all $s\in\zo$. For $s\in\zo$ and $\left|t\right|>0$, the cosine of the angle $\theta_s(t)$ between the vectors $\alpha_s(t)$ and $\alpha_s'(t)$ is
\begin{equation*}
\dfrac{\alpha_s(t)}{\left\|\alpha_s(t)\right\|}\cdot\dfrac{\alpha_s'(t)}{\left\|\alpha_s'(t)\right\|}
\end{equation*}
which becomes 
\begin{equation}\label{eq20}
\pm\dfrac{\ba_d(s)+\frac{1}{t}\ba_{d-1}(s)+\cdots+\frac{1}{t^d}\ba_0(s)}{\left\|\ba_d(s)+\frac{1}{t}\ba_{d-1}(s)+\cdots+\frac{1}{t^d}\ba_0(s)\right\|}\cdot\dfrac{d\hp\ba_d(s)+\frac{d-1}{t}\hp\ba_{d-1}(s)+\cdots+\frac{1}{t^{d-1}}\hp\ba_1(s)}{\left\|d\hp\ba_d(s)+\frac{d-1}{t}\hp\ba_{d-1}(s)+\cdots+\frac{1}{t^{d-1}}\hp\ba_1(s)\right\|}
\end{equation}
accordingly as $t$ is positive or negative. Thus, we get
\begin{align}
\left|\cos(\theta_s(t))\right|&\geq\dfrac{\left|d\left\|\ba_d(s)\right\|^2+\frac{1}{t}\hp f_1(s)+\frac{1}{t^2}\hp f_2(s)+\cdots+\frac{1}{t^{2d-1}}\hp f_{2d-1}(s)\right|}{d\left\|\ba_d(s)\right\|^2+\frac{1}{\left|t\right|}\hp h_1(s)+\frac{1}{\left|t^2\right|}\hp h_2(s)+\cdots+\frac{1}{\left|t^{2d-1}\right|}\hp h_{2d-1}(s)}\nonumber\\
&\geq\dfrac{\left|d\left\|\ba_d(s)\right\|^2-\left|\frac{1}{t}\hp f_1(s)+\frac{1}{t^2}\hp f_2(s)+\cdots+\frac{1}{t^{2d-1}}\hp f_{2d-1}(s)\right|\ho\right|}{d\left\|\ba_d(s)\right\|^2+\frac{1}{\left|t\right|}\hp h_1(s)+\frac{1}{\left|t^2\right|}\hp h_2(s)+\cdots+\frac{1}{\left|t^{2d-1}\right|}\hp h_{2d-1}(s)}\hs,\label{eq21}
\end{align}
where for each $i\in\{\hp1,2,\ldots,2d-1\hp\}$,
\begin{align*}
f_i(s)&=\sum_{j,k}N_{ijk}\ho\ba_j(s)\cdot\ba_k(s)\quad\mbox{and}\\ 
h_i(s)&=\sum_{j,k}M_{ijk}\left\|\ba_j(s)\right\|\left\|\ba_k(s)\right\|
\end{align*}
for $s\in\zo$. Note that $N_{ijk}$'s and $M_{ijk}$'s are nonnegative integers. Since $\ba_d$ is continuous and $\ba_d(s)\neq0$ for all $s\in\zo$, there exist some positive real numbers $m$ and $M$ such that 
\begin{equation}\label{eq22}
m\leq\left\|\ba_d(s)\right\|\leq M
\end{equation} 
for all $s\in\zo$. Note that $f_i$'s and $h_i$'s are continuous and bounded real valued functions, so for some $r_0\gg0$, we have 
\begin{align}
\left|\frac{1}{t}\hp f_1(s)+\frac{1}{t^2}\hp f_2(s)+\cdots+\frac{1}{t^{2d-1}}\hp f_{2d-1}(s)\right|&\leq\frac{d\hp m^2}{2}\quad\mbox{and}\label{eq23}\\[3pt]
\frac{1}{\left|t\right|}\hp h_1(s)+\frac{1}{\left|t^2\right|}\hp h_2(s)+\cdots+\frac{1}{\left|t^{2d-1}\right|}\hp h_{2d-1}(s)&\leq d\hp M^2\label{eq24}
\end{align}
for all $\left|t\right|\geq r_0$ and for all $s\in\zo$. Using Inequalities (\ref{eq22})-(\ref{eq24}) in Expression (\ref{eq21}) gives
\begin{align*}
\left|\cos(\theta_s(t))\right|\geq\dfrac{d\hp m^2-\frac{d\hp m^2}{2}}{d\hp M^2+ d\hp M^2}\geq\dfrac{m^2}{4\hp M^2}>0
\end{align*}
for all $\left|t\right|\geq r_0$ and for all $s\in\zo$. This proves the first statement. Also, for any $s\in\zo$, Expression (\ref{eq20}) approaches to
\begin{equation*}
\pm\dfrac{\ba_d(s)\cdot d\hp\ba_d(s)}{\left\|\ba_d(s)\right\|\left\|d\hp\ba_d(s)\right\|}=\pm1
\end{equation*}
accordingly as $t\to\pm\infty$. In other words, the angle $\theta_s(t)$ between the vectors $\alpha_s(t)$ and $\alpha_s'(t)$ approaches $0$ as $t\to+\infty$ and it approaches $\pi$ as $t\to-\infty$.
\end{proof}

\subsection{Real semialgebraic sets}\label{sec2.3}

Algebraic geometry is the study of algebraic sets which are the sets of zeros of polynomials. When the polynomials are over the field of real numbers, the notion of algebraic sets can be extended to a bigger class of sets known as semialgebraic sets. This includes the sets of points which are solutions of some inequalities satisfied by polynomial functions. More precisely, the semialgebraic sets are defined as follows: 
\begin{definition}
Semialgebraic subsets of $\mathbb{R}^n$ form a smallest class $\mathcal{S}_n$ of subsets of $\mathbb{R}^n$ such that:
\begin{enumerate}[(1)]
\item If $P\in\mathbb{R}[X_1,\ldots,X_n]$, then $\big\{\hp x\in\mathbb{R}^n : P(x)=0\hp\big\}$ and $\big\{\hp x\in\mathbb{R}^n : P(x)>0\hp\big\}$ are in $\mathcal{S}_n$.
\item If $A\in\mathcal{S}_n$ and $B\in\mathcal{S}_n$, then $A\cup B$, $A\cap B$ and  $\mathbb{R}^n\setminus A$ are in $\mathcal{S}_n$.
\end{enumerate}
\end{definition}

For example, the lower region $A=\{(x,y)\in\rw\mid y\leq x^2-9\}$ of the plane which is bounded by the parabola $y=x^2-9$ is a semialgebraic subset of $\rw$. Also, a union of the upper half space $B=\{(x,y,z)\in\rt\mid z>0\}$ and the closed unit ball $C=\{(x,y,z)\in\rt\mid x^2+y^2+z^2\leq0\}$ is a semialgebraic subset of $\rt$.

A detailed exposure to the study of real semialgebraic sets can be found in \cite{mc} and \cite{br}. Some of the important results (proofs can be seen in \cite{mc} and \cite{br}) that will be used in this paper are stated below: 

\begin{theorem}[Tarski-Seidenberg: Second Form]\label{th37}
Let $A$ be a semialgebraic subset of $\mathbb{R}^{n+1}$ and let $\pi:\mathbb{R}^{n+1}\to\mathbb{R}^n$ be the projection onto the space of the first $n$ coordinates. Then $\pi(A)$ is a semialgebraic subset of $\mathbb{R}^n$.
\end{theorem} 

\begin{corollary}\label{th38}
If $A$ is a semialgebraic subset of $\mathbb{R}^{n+k}$, then its image by the projection onto the space of the first $n$ coordinates is a semialgebraic subset of $\mathbb{R}^n$.
\end{corollary}

\begin{theorem}\label{th39}
A semialgebraic set has only finitely many connected components and all they are semialgebraic.
\end{theorem}

\begin{proposition}\label{th40} 
A connected semialgebraic set is path connected.
\end{proposition}

\section{Spaces of Polynomial Knots}\label{sec3}

Let us denote the set of all polynomial maps by $\spm$ and the set of all polynomial knots by $\spk$. It is easy to see that $\spk$ is a proper subset of $\spm$. For an integer $d\geq 2$, we define the following sets: 
\begin{align*}
\ad&=\big\{\fgh\in\spm\mid\deg(f)\leq d-2,\ho\deg(g)\leq d-1\;\mbox{and}\;\deg(h)\leq d\big\},\\[3pt]
\bd&=\big\{\fgh\in\spm\mid\deg(f)<\deg(g)<\deg(h)\leq d\big\},\\[3pt]
\cd&=\big\{\fgh\in\spm\mid\deg(f)=d-2,\hf\deg(g)=d-1\;\mbox{and}\;\deg(h)=d\big\}\hp.
\end{align*}
Also, let $\od=\spk\cap\ad\hp,\;\pd=\spk\cap\bd$ and $\qd=\spk\cap\cd$. The set $\ad$ can be identified with the Euclidean space $\rtd$. In fact, there is a natural bijection $\eta:\ad\rightarrow\rtd$ given by 
\begin{equation}\label{eq14}
\fgh\mapsto\abc,
\end{equation} 
where $a_i,\,b_i$ and $c_i$, for $i=0,1,\ldots, d$, are coefficients of $t^i$ in the polynomials $f,\,g$ and $h$ respectively. We have a metric $\rho$ on $\ad$ given by
\begin{equation}\label{eq16}
\rho\hp(\phi,\psi)= \xi\hp\big(\eta(\phi),\eta(\psi)\big)
\end{equation}
for $\phi,\psi\in\ad$, where $\xi$ denotes the Euclidean distance in $\rtd$. With this metric, the set $\ad$ becomes a topological space. Obviously, the map $\eta$ is a homeomorphism (in fact, a diffeomorphism) between the topological spaces $\ad$ and $\rtd$. With respect to the subspace topology the sets $\bd,\hp \cd,\hp \od,\hp \pd$\, and\, $\qd$\, become subspaces of the space $\ad$ and hence they can be identified with the subspaces of $\rtd$. Note that the spaces $\od,\pd$ and $\qd$ are the spaces of polynomial knots of degree at most $d$.
\begin{remark}
Using a suitable invertible linear transformation, a polynomial knot of degree $d\geq2$ can be transformed to a polynomial knot $t\mapsto\fght$ such that $\deg(f)\leq d-2,\ho\deg(g)\leq d-1$ and $\deg(h)\leq d$. This shows that every knot-type can be represented by a polynomial knot belonging to the space $\od$ for some $d\geq2$.
\end{remark}

Note that $\spm=\mcup_{d\geq2}\ad$\: and\: $\spk=\mcup_{d\geq2}\od$. Therefore, the sets $\spm$ and $\spk$ can be given the inductive limit topology; that is, a set $U\subseteq\spm$ is open in $\spm$ if and only if the set $U\cap\ad$ is open in $\ad$ for all $d\geq2$, and a set $V\subseteq\spk$ is open in $\spk$ if and only if the set $V\cap\od$ is open in $\od$ for all $d\geq2$.\vskip0.1mm
It is easy to check that $\mathcal{B}_2=\mathcal{C}_2=\mathcal{P}_2=\mathcal{Q}_2$ and $\mathcal{C}_3=\mathcal{Q}_3$. Also, it is obvious that $\cd\nsubseteq\mathcal{C}_{d+1}$ and $\qd\nsubseteq\mathcal{Q}_{d+1}$ for all $d\geq2$. Furthermore, for the collection $\mathfrak{F}=\big\{\spm,\hp\spk,\hp\ad,\hp\bd,\hp\cd,\hp\od,\hp\pd,\hp\qd\mid d\geq2\big\}$, with respect to the strict set inclusion as partial order, we have the Hesse diagram as given in Figure \ref{fig1}.
\begin{figure}[H]
\begin{center}\begin{tikzpicture}[scale=0.65]
\node (aw)  at  (-5,-1.5)  {$\mathcal{A}_2$};
\node (ow)  at  (-5,-3)    {$\mathcal{O}_2$};
\node (at)  at  (-3,0)     {$\mathcal{A}_3$};
\node (ot)  at  (-3,-1.5)  {$\mathcal{O}_3$};
\node (qw)  at  (-3,-4.5)  {$\mathcal{B}_2=\mathcal{C}_2=\mathcal{P}_2=\mathcal{Q}_2$};
\node (ad)  at  (-1,1.5)   {$\mathcal{A}_m$};
\node (od)  at  (-1,0)     {$\mathcal{O}_m$};
\node (bt)  at  (-1,-1.5)  {$\mathcal{B}_3$};
\node (pt)  at  (-1,-3)    {$\mathcal{P}_3$};
\node (adp) at  (1,3)      {$\mathcal{A}_n$};
\node (odp) at  (1,1.5)    {$\mathcal{O}_n$};
\node (bd)  at  (1,0)      {$\mathcal{B}_m$};
\node (pd)  at  (1,-1.5)   {$\mathcal{P}_m$};
\node (qt)  at  (1,-4.5)   {$\mathcal{C}_3=\mathcal{Q}_3$};
\node (a)   at  (3,4.5)    {$\spm$};
\node (p)   at  (3,3)      {$\spk$};
\node (bdp) at  (3,1.5)    {$\mathcal{B}_n$};	
\node (pdp) at  (3,0)      {$\mathcal{P}_n$};
\node (cd)  at  (3,-1.5)   {$\mathcal{C}_m$};
\node (qd)  at  (3,-3)     {$\mathcal{Q}_m$};
\node (cdp) at  (5,0)      {$\mathcal{C}_n$};
\node (qdp) at  (5,-1.5)   {$\mathcal{Q}_n$};
	
\draw (ow)--(ot)--(od)--(odp)--(p)--(a)--(adp)--(ad)--(at)--(aw)--(ow)--(qw)--(pt)--(bt)--(bd)--(bdp)--(cdp)--(qdp)--(pdp)--(pd)--(qd)--(cd) (at)--(ot)--(pt)--(qt) (ad)--(od) (adp)--(odp) (bd)--(pd)--(pt) (bdp)--(pdp);
 
\draw[preaction={draw=white, -,line width=4pt}](at)--(bt) (od)--(pd) (ad)--(bd)--(cd) (odp)--(pdp) (adp)--(bdp);
\end{tikzpicture}\vo
for $n>m>3$
\caption{Partial order in $\mathfrak{F}$}
\label{fig1}
\end{center}
\end{figure}

\begin{proposition}\label{th1}
The space $\cd$, for $d\geq2$, is open in the space $\ad$.
\end{proposition} 

\begin{proof}
The set $\eta\hp(\mathcal{C}_2)$ is written as follows:
\begin{equation*}
\eta\hp(\mathcal{C}_2)=\big\{\hs(\hf a_0,b_0,b_1,c_0,c_1,c_2 \hf) \in\mathbb{R}^6\mid b_1 c_2\neq0 \hs \big\}.
\end{equation*} 
Also, for $d\geq3$, we have
\begin{equation*}
\eta\hp(\cd)=\big\{\hs\abc\in\rtd\mid a_{d-2}b_{d-1} c_d\neq0 \hs\big\}.
\end{equation*}
For $d\geq2$, the set $\eta\hp(\cd)$ is an open subset of $\eta\hp(\ad)=\rtd$. Since $\eta$ is a homeomorphism, the space $\cd$ is open in the space $\ad$. 
\end{proof}

\begin{remark}\label{rm7}
The sets $\mathcal{B}_2, \mathcal{P}_2$ and $\mathcal{Q}_2$ all are equal to $\mathcal{C}_2$, and hence they are open in the space $\mathcal{A}_2$.
\end{remark}

\begin{lemma}\label{th24}
For any nonzero real numbers $a$ and $b$, and an even integer $n\geq2$, we have the following:
\begin{enumerate}[(1)]
\item $0<\frac{1}{a^n+a^{n-1}b+\cdots+b^n}\leq\max\left\{\hp \frac{1}{a^n},\frac{1}{b^n}\hp\right\}$.\vw
\item$0<\frac{a^k+a^{k-1}b+\cdots+b^k}{a^n+a^{n-1}b+\cdots+b^n}\leq\max\left\{\frac{1}{a^{n-k}},\frac{1}{b^{n-k}}\right\}$,\hskip1.5mm for an even integer $k\in\{1,2,\ldots,n\}$.\vw
\item $\left|\frac{a^k+a^{k-1}b+\cdots+b^k}{a^n+a^{n-1}b+\cdots+b^n}\right| < \min\left\{\frac{1}{\left|a^{n-k}\right|},\frac{1}{\left|b^{n-k}\right|}\right\}$,\hskip1.5mm for an odd integer $k\in\{1,2,\ldots,n\}$.
\end{enumerate}
\end{lemma}

\begin{proof}
Without loss of generality we assume that $\lvert a\rvert\geq\lvert b\rvert$.\vskip0.1mm
(1) Since $\lvert a\rvert\geq\lvert b\rvert$, we have $a+b\geq0$ if $a>0$ and $a+b\leq0$ if $a<0$. Also, an odd power of $a$ and an even power of $b$ occur in each term of the expression $y=a^{n-1}+a^{n-3}b^2+\cdots+ab^{n-2}$, so\, $y\geq0$\, if\, $a>0$ and\,  $y\leq0$\, if\, $a<0$. Thus, in either case $(a+b)\left(a^{n-1}+a^{n-3}b^2+\cdots+ab^{n-2}\right)=(a+b)y\geq0$ and hence
\begin{equation*}
a^n+a^{n-1}b+\cdots+b^n=(a+b)\left(a^{n-1}+a^{n-3}b^2+\cdots+ab^{n-2}\right)+b^n\geq b^n>0\,.
\end{equation*} 
In other words, we have $0<\frac{1}{a^n+a^{n-1}b+\cdots+b^n}\leq\frac{1}{b^n}=\max\left\{\frac{1}{a^n},\frac{1}{b^n}\right\}$.\vskip0.7mm
(2) Let an even integer $k\in\{1,2,\ldots,n\}$ be given. By the first part, we have $a^k+a^{k-1}b+\cdots+b^k>0$. For $k=n$, the inequality is trivially true, so assume $k<n$. We now consider the following expression:
\begin{align}
x_k&=\frac{a^n+a^{n-1}b+\cdots+b^n}{a^k+a^{k-1}b+\cdots+b^k}\label{eq10}\\
&=\frac{a^n+a^{n-1}b+\cdots+a^{k+1}b^{n-k-1}}{a^k+a^{k-1}b+\cdots+b^k}+b^{n-k}\label{eq11}\\
&=\frac{(a+b)\left(a^{n-1}+a^{n-3}b^2+\cdots+a^{k+1}b^{n-k-2}\right)}{a^k+a^{k-1}b+\cdots+b^k}+b^{n-k}\label{eq12}.
\end{align}
An odd power of $a$ and an even power of $b$ occur in each term of the expression $y_k=a^{n-1}+a^{n-3}b^2+\cdots+a^{k+1}b^{n-k-2}$,\hskip1.2mm so\hskip1.2mm $y_k\geq0$\hskip1.2mm if\hskip1.2mm $a>0$ and\hskip1.2mm  $y_k\leq0$\hskip1.2mm if\hskip1.2mm $a<0$. Recall that $a+b\geq0$\hskip1.2mm if\hskip1.2mm $a>0$ and\hskip1.2mm  $a+b\leq0$\hskip1.2mm if\hskip1.2mm $a<0$. Hence in either case, the numerator of the first term of Expression (\ref{eq12}) is non-negative. Since $a^k+a^{k-1}b+\cdots+b^k>0$, the first term of Expression (\ref{eq12}) is non-negative. Thus, we have $x_k\geq b^{n-k}>0$\hskip1.5mm and hence $0<\frac{1}{x_k}\leq\frac{1}{b^{n-k}}=\max\left\{\frac{1}{a^{n-k}},\frac{1}{b^{n-k}}\right\}$.\vskip0.7mm
(3) Let $k\in\{1,2,\ldots,n\}$ be an odd integer. If $a^k+a^{k-1}b+\cdots+b^k=0$, then the inequality is trivially true, so assume $a^k+a^{k-1}b+\cdots+b^k\neq0$. Now consider the following expression:
\begin{align}
x_k&=\frac{a^n+a^{n-1}b+\cdots+b^n}{a^k+a^{k-1}b+\cdots+b^k}\label{eq7}\\
&=a^{n-k}+\frac{a^{n-k-1}b^{k+1}+a^{n-k-2}b^{k+2}+\cdots+b^n}{a^k+a^{k-1}b+\cdots+b^k}\label{eq8}\\
&=a^{n-k}+\frac{(a+b)\left(a^{n-k-2}b^{k+1}+a^{n-k-4}b^{k+3}+\cdots+ab^{n-2}\right)+b^n}{(a+b)\left(a^{k-1}+a^{k-3}b^2+\cdots+b^{k-1}\right)}\label{eq9}.
\end{align}
Note that an odd power of $a$ and an even power of $b$ occur in each term of the expression\, $a^{n-k-2}b^{k+1}+a^{n-k-4}b^{k+3}+\cdots+ab^{n-2}$,\, so by the similar argument as in the second part,\, $(a+b)\left(a^{n-k-2}b^{k+1}+a^{n-k-4}b^{k+3}+\cdots+ab^{n-2}\right)$ is non-negative, and hence the numerator of the second term of Expression (\ref{eq9}) is positive. Also, $a$ and $b$ occur with even powers in each term of the expression $a^{k-1}+a^{k-3}b^2+\cdots+b^{k-1}$,\hskip1.5mm so it is positive. Therefore, the sign of the second term of Expression (\ref{eq9}) is same as the sign of $a+b$ and hence it is same as the sign of $a$. This shows that $x_k> a^{n-k}>0$\, if $a>0$\, and\, $x_k < a^{n-k}<0$\, if\, $a<0$ and hence we have $\frac{1}{\left| x_k\right|}<\frac{1}{\left| a^{n-k}\right|}=\min\left\{\frac{1}{\left| a^{n-k}\right|},\frac{1}{\left| b^{n-k}\right|}\right\}$.
\end{proof}

\begin{lemma}\label{th25}
Let $n\geq2$ be a fixed even integer, and for $i\in\{0,1,\ldots, n\}$, let $\big\{\alpha_{ij}\big\}_{j=1}^\infty$ be a sequence of real numbers which converges to $\alpha_i$. Suppose $\alpha_n\neq0$, and let $\big\{s_j\big\}_{j=1}^\infty$ and $\big\{t_j\big\}_{j=1}^\infty$ be sequences of real numbers such that $s_j\to\pm\infty$\ho and $t_j\to\pm\infty$\ho as $j\to\infty$. Then 
\begin{equation*}
\lim_{j\to\infty}\left(\alpha_{0j}+\alpha_{1j}\left(s_j+t_j\right)+\cdots+\alpha_{nj}\left(s_j^n+s_j^{n-1}t_j+\cdots+t_j^n\right)\right)=\pm\infty
\end{equation*} 
depending on whether $\alpha_n$ is positive or negative.
\end{lemma}

\begin{proof}
Since $s_j\to\pm\infty$ and $t_j\to\pm\infty$ as $j\to\infty$, we can choose $N\in\mathbb{N}$ such that $s_j\neq0$ and $t_j\neq0$ for all $j\geq N$. For $j\geq N$, let
\begin{equation}\label{eq13}
z_j=\frac{\alpha_{0j}+\alpha_{1j}\left(s_j+t_j\right)+\cdots+\alpha_{nj}\left(s_j^n+s_j^{n-1}t_j+\cdots+t_j^n\right)}{s_j^n+s_j^{n-1}t_j+\cdots+t_j^n}\hs.
\end{equation}
The expression for $z_j$ can be rewritten as follows:
\begin{align}\label{eq17}
z_j&=\alpha_{0j}\left(\dfrac{1}{s_j^n+s_j^{n-1}t_j+\cdots+t_j^n}\right)+\alpha_{1j}\left(\dfrac{s_j+t_j}{s_j^n+s_j^{n-1}t_j+\cdots+t_j^n}\right)\nonumber\\[3pt]
&\hskip4mm+\cdots+\alpha_{n-1,j}\left(\dfrac{s_j^{n-1}+s_j^{n-2}t_j+\cdots+t_j^{n-1}}{s_j^n+s_j^{n-1}t_j+\cdots+t_j^n}\right)+\alpha_{nj}\hp.
\end{align}
By Lemma \ref{th24}, for $i\in\{0,1,\ldots,n-1\}$ and $j\geq N$, we have
\begin{align}
&\hskip0.9mm 0\leq\frac{1}{s_j^n+s_j^{n-1}t_j+\cdots+t_j^n}\leq\max\left\{\frac{1}{s_j^n},\frac{1}{t_j^n}\right\}\quad\mbox{and}\label{eq25}\\
&\left|\frac{s_j^i+s_j^{i-1}t_j+\cdots+t_j^i}{s_j^n+s_j^{n-1}t_j+\cdots+t_j^n}\right|\leq\max\left\{\frac{1}{\rvert s_j^{n-i}\rvert},\frac{1}{\lvert t_j^{n-i}\rvert}\right\}\hf.\hp\label{eq26}
\end{align}
For $i\in\{0,1,\ldots,n-1\}$, note that $\frac{1}{s_j^{n-i}}\to0$ and $\frac{1}{t_j^{n-i}}\to0$ as $j\to\infty$, so by Eqs. (\ref{eq25}) and (\ref{eq26}), we have	
\begin{equation*}	
\lim_{j\to\infty}\frac{1}{s_j^n+s_j^{n-1}t_j+\cdots+t_j^n}=0 \quad\mbox{and}\quad\lim_{j\to\infty}\frac{s_j^i+s_j^{i-1}t_j+\cdots+t_j^i}{s_j^n+s_j^{n-1}t_j+\cdots+t_j^n}=0\hs.
\end{equation*}
For $i\in\{0,1,\ldots,n\}$, note that $\alpha_{ij}\to\alpha_i$ as $j\to\infty$, so by Eq. (\ref{eq17}), we have $z_j\to\alpha_n$ as $j\to\infty$ (note that $\alpha_n\neq0$). By Lemma \ref{th24}, we have $s_j^n+s_j^{n-1}t_j+\cdots+t_j^n\geq\min\hp\left\{s_j^n,t_j^n\right\}$ for all $j\geq N$, so $s_j^n+s_j^{n-1}t_j+\cdots+t_j^n\to\infty$ as $j\to\infty$\hf. Hence by Eq. (\ref{eq13}), we get the required result.
\end{proof}

\begin{lemma}\label{th7}
Let $\gamma(t)=\alpha_0+\alpha_1 t+\cdots+\alpha_k t^k$ be a polynomial in a variable $t$ and let $\Gamma(s, t)=\alpha_1+\alpha_2\left(s+t\right) +\cdots+\alpha_k\left(s^{k-1}+s^{k-2}t+\cdots+t^{k-1}\right)$ be a polynomial in two variables $s$ and $t$. Then a point $(s_0, t_0)\in\rw$ is a zero of\hw$\Gamma$\ho if and only if either $\gamma(s_0)=\gamma(t_0)$ or $\gamma'(t_0)=0$ accordingly as $s_0\neq t_0$ or $s_0=t_0$.
\end{lemma} 

\begin{proof}
It is easy to check the following:
\begin{align}
\Gamma(s, t)&=\dfrac{\gamma(s)-\gamma(t)}{s-t}\quad \mbox{for all}\; s,t\in\ro\; \mbox{with}\; s\neq t,\; \mbox{and}\label{eq18}\\
\Gamma(t, t)&=\gamma'(t)\quad \mbox{for all}\; t\in\ro.\label{eq19}
\end{align}\vskip0.1mm
Suppose a point $(s_0, t_0)\in\rw$ is a zero of $\Gamma$. If $s_0\neq t_0$, then by Eq. (\ref{eq18}), $\gamma(s_0)=\gamma(t_0)$. If $s_0=t_0$, then by Eq. (\ref{eq19}), $\gamma'(t_0)=0$.\vskip0.1mm
If $\gamma(s_0)=\gamma(t_0)$ for some $(s_0, t_0)\in\rw$ with $s_0\neq t_0$, then by Eq. (\ref{eq18}), the point $(s, t)=(s_0, t_0)$ is a zero of $\Gamma$. If $\gamma'(t_0)=0$ for some $t_0\in\ro$, then by Eq. (\ref{eq19}), the point $(s, t)=(t_0, t_0)$ is a zero of $\Gamma$.
\end{proof}

\begin{lemma}\label{th8}
Let $f(t)= \at,\ho g(t)=\bt$ and $h(t)=\ct$ be real polynomials. Then the polynomial map $\phi:\ro\to\rt$ given by $t\mapsto\fght$ is a smooth embedding if and only if the polynomials
\begin{align*}
F(s, t)&=\past,\\
G(s, t)&=\pbst\hw\mbox{and}\\
H(s, t)&=\pcst
\end{align*}
do not have a common zero.
\end{lemma} 

\begin{proof}
Suppose that $(s_0, t_0)\in\rw$ is a common zero of the polynomials $F$, $G$ and $H$. If $s_0\neq t_0$, then by Lemma \ref{th7}, $f(s_0)=f(t_0)$, $g(s_0)=g(t_0)$ and $h(s_0)=h(t_0)$, and hence $\phi(s_0)=\phi(t_0)$. If $s_0=t_0$, then again by Lemma \ref{th7}, $f'(t_0)=g'(t_0)=h'(t_0)=0$\hp; that is, $\phi'(t_0)=0$. Thus, in either case, $\phi$ is not an embedding.\vskip0.1mm
To prove the converse, assume that $\phi$ is not an embedding. Then we have $\phi(s_0)=\phi(t_0)$ for some $s_0\neq t_0$, or $\phi'(u_0)=0$\, for some $u_0\in\ro$. In other words, $f(s_0)=f(t_0)$, $g(s_0)=g(t_0)$ and $h(s_0)=h(t_0)$ for some $s_0\neq t_0$, or $f'(u_0)=g'(u_0)=h'(u_0)=0$ for some $u_0\in\ro$.  Therefore, by Lemma \ref{th7}, the polynomials $F$, $G$ and $H$ have a common zero.
\end{proof}
 
\begin{theorem}\label{th10}
The space $\qd$, for $d\geq2$, is open in the space $\cd$.
\end{theorem}

\begin{proof}
Since $\mathcal{Q}_2=\mathcal{C}_2$ and $\mathcal{Q}_3=\mathcal{C}_3$, the theorem is trivially true for $d=2,3$. So we assume that $d\geq4$. We show that $\cd\setminus\qd$ is closed in $\cd$. Let $\left\{\phi_j\right\}_{j=1}^{\infty}$ be a sequence of points in $\cd\setminus\qd$ which converges to a point $\phi$ in $\cd$. We need to show that $\phi\in\cd\setminus\qd$. For $j\in\mathbb{N}$, let $\phi_j$ be given by $t\mapsto\big(f_j(t),\hp g_j(t),\hp h_j(t)\big)$, and its component polynomials be given by $f_j(t)=a_{0j}+a_{1j} t+\cdots+a_{d-2,j}\hp t^{d-2},\, g_j(t)=b_{0j}+b_{1j} t+\cdots+b_{d-1,j}\hp t^{d-1}$ and $h_j(t)=c_{0j}+c_{1j} t+\cdots +c_{dj}\hp t^d$. Also, let $\phi$ be given by $t\mapsto\fght$, and the component polynomials $f,\hp g$ and  $h$ be given by $f(t)=a_0+a_1 t+\cdots+a_{d-2}\hp t^{d-2},\; g(t)=b_0+b_1 t+\cdots+b_{d-1}\hp t^{d-1}$ and $h(t)=c_0+c_1 t+\cdots +c_d\hp t^d$. For $j\in\mathbb{N}$, since $\phi_j$ is not a smooth embedding, by Lemma \ref{th8}, we can choose $(s_j,\hp t_j)\in\rw $ such that
\begin{align}
& a_{1j}+a_{2j}\left(s_j+t_j\right)+\cdots+a_{d-2,j}\left(s_j^{d-3}+s_j^{d-4}t_j+\cdots+t_j^{d-3}\right)=0,\label{eq1}\\
& b_{1j}+b_{2j}\left(s_j+t_j\right)+\cdots+b_{d-1,j}\left(s_j^{d-2}+s_j^{d-3}t_j+\cdots+t_j^{d-2}\right)=0\quad\mbox{and}\label{eq2}\\
& c_{1j}+c_{2j}\left(s_j+t_j\right)+\cdots+c_{dj}\left(s_j^{d-1}+s_j^{d-2}t_j+\cdots+t_j^{d-1}\right)=0\hp.\label{eq3}
\end{align}\vskip0.1mm 
We claim that the sequence $\big\{(s_j,\hp t_j)\big\}_{j=1}^\infty$ is bounded. Suppose on the contrary that this sequence is unbounded, then at least one of the sequence $\big\{s_j\big\}_{j=1}^\infty$ or $\big\{t_j\big\}_{j=1}^\infty$ is unbounded. We may assume that the sequence $\big\{s_j\big\}_{j=1}^\infty$ is unbounded, so it has a subsequence which diverges to $\pm\infty$. Again, without loss of generality, we may assume that the sequence $\big\{s_j\big\}_{j=1}^\infty$ itself diverges to $\pm\infty $. There are two cases according to which either $\big\{t_j\big\}_{j=1}^\infty$ is bounded or not.\vskip0.1mm
(1) If the sequence $\big\{t_j\big\}_{j=1}^\infty$ is bounded: For $i=0,1,\ldots, d$, note that $c_{ij}\to c_i$ as $j\to\infty$ and $c_d\neq0$. Since the sequence $\big\{s_j\big\}_{j=1}^\infty$ diverges to $\pm\infty $, we have 
\begin{equation*}
\lim_{j\to\infty}\left(c_{1j}+c_{2j}\left(s_j+t_j\right)+\cdots+c_{dj}\left(s_j^{d-1}+s_j^{d-2}t_j+\cdots+t_j^{d-1}\right)\right)=\pm\infty\,.
\end{equation*}
This is a contradiction, since the right hand side of Eq. (\ref{eq3}) is zero for all $j\in\mathbb{N}$.\vskip0.1mm
(2) If the sequence $\big\{t_j\big\}_{j=1}^\infty$ is unbounded: In this case, the sequence $\big\{t_j\big\}_{j=1}^\infty$ has a subsequence which diverges to $\pm\infty$. Without loss of generality, we may assume that the sequence $\big\{t_j\big\}_{j=1}^\infty$ itself diverges to $\pm\infty $. For $i=0,1,\ldots, d$, note that $b_{ij}\to b_i$ and $c_{ij}\to c_i$ as $j\to\infty$, and $b_{d-1}\neq0$ and $c_d\neq0$. Also, note that the sequence $\big\{s_j\big\}_{j=1}^\infty$\hf diverges to $\pm\infty$, so by Lemma \ref{th25}, we have
\begin{align*} 
&\lim_{j\to\infty}\left(b_{1j}+b_{2j}\left(s_j+t_j\right)+\cdots+b_{j,d-1}\left(s_j^{d-2}+s_j^{d-3}t_j+\cdots+t_j^{d-2}\right)\right)=\pm\infty\quad\mbox{or}\\ 
&\lim_{j\to\infty}\left(c_{1j}+c_{2j}\left(s_j+t_j\right)+\cdots+c_{dj}\left(s_j^{d-1}+s_j^{d-2}t_j+\cdots+t_j^{d-1}\right)\right)=\pm\infty\,.
\end{align*}
This is again a contradiction, since the right hand sides of Eqs. (\ref{eq2}) and (\ref{eq3}) are zero for all $j\in\mathbb{N}$.\vskip0.1mm
The cases (1) and (2) both together show that the sequence  $\big\{(s_j,\hp t_j)\big\}_{j=1}^\infty$ is bounded and hence it has a subsequence which is convergent. We may assume that the sequence $\big\{(s_j,\hp t_j)\big\}_{j=1}^\infty$ itself is convergent, say converges to a point $ (s_0, t_0)\in\rw$. Since the sequence
\begin{equation*}
\big\{(a_{0j}, a_{1j}, \ldots, a_{d-2,j}, b_{0j}, b_{1j}, \ldots, b_{d-1,j}, c_{0j}, c_{1j}, \ldots, c_{dj})\big\}_{j=1}^\infty
\end{equation*}
converges to $(a_0, a_1, \ldots, a_{d-2}, b_0, b_1, \ldots, b_{d-1}, c_0, c_1, \ldots, c_d)$, the left hand sides of Eqs. (\ref{eq1}), (\ref{eq2}) and (\ref{eq3}) converge respectively to
\begin{align}
& a_1+a_2\left(s_0+t_0\right)+\cdots+a_{d-2}\left(s_0^{d-3}+s_0^{d-4}t_0+\cdots+t_0^{d-3}\right),\label{eq4}\\
& b_1+b_2\left(s_0+t_0\right)+\cdots+b_{d-1}\left(s_0^{d-2}+s_0^{d-3}t_0+\cdots+t_0^{d-2}\right)\quad\mbox{and}\label{eq5}\\
& c_1+c_2\left(s_0+t_0\right)+\cdots+c_d\left(s_0^{d-1}+s_0^{d-2}t_0+\cdots+t_0^{d-1}\right)\hp.\label{eq6}
\end{align}
The right hand sides of Eqs. (\ref{eq1}), (\ref{eq2}) and (\ref{eq3}) are zero, and hence so are Expressions (\ref{eq4}), (\ref{eq5}) and (\ref{eq6}). Therefore, by Lemma \ref{th8}, the map $\phi$ is not an embedding. In other words, the map $\phi$ is an element of the space $\cd\setminus\qd$.
\end{proof}

The following corollary follows trivially from Proposition \ref{th1} and Theorem \ref{th10}.
 
\begin{corollary}\label{th11} 
The space $\qd$, for $d\geq2$, is open in the space $\ad$.
\end{corollary}

\begin{remark}
For $d\geq2$, we have an element $\phi\in\od$ given by $t\mapsto(\hp 1,\, 1,\, t\hp)$. For $\epsilon>0$, an open ball $V(\phi,\epsilon)=\big\{\tau\in\ad\mid\rho(\phi,\tau)<\epsilon\big\}$ (where $\rho$ is the metric as defined in Eq. (\ref{eq16})) contains an element $\phi_\epsilon\in\ad$ given by 
\begin{equation*}
t\mapsto\big(\hp 1,\, 1,\, t-(\epsilon/2) t^2\hp\big)
\end{equation*} 
which is not an embedding (since $\phi_\epsilon'(1/ \epsilon)=0$). Thus $\phi$ is not an interior point of $\od$. This shows that {\it the space $\od$, for $d\geq2$, is not open in the space $\ad$}.
\end{remark}

\begin{remark}
For $d\geq3$, let a map $\psi\in\bd$ be given by $t\mapsto(\hp t^{d-3},\, t^{d-2},\, t^{d-1}\hp)$. For $\epsilon>0$, an open ball $V(\psi,\epsilon)$ contains an element $\psi_\epsilon\in\ad$ given by
\begin{equation*}
t\mapsto\big(\hp t^{d-3},\, t^{d-2}+(\epsilon/2)t^{d-1},\, t^{d-1}\hp\big)
\end{equation*}
which is not an element of $\bd$. Therefore $\psi$ is not an interior point of the space $\bd$, and hence {\it the space $\bd$, for $d\geq3$, is not open in the space $\ad$}.
\end{remark} 

\begin{remark}
For $d\geq3$, we have an element $\sigma\in\pd$ given by $t\mapsto(\hp1,\,t,\,t^2\hp)$. For $\epsilon>0$, an open ball $V(\sigma,\epsilon)$ contains an element $\sigma_\epsilon\in\bd$ given by
\begin{equation*}
t\mapsto\big(\hp1,\,t-(\epsilon/2)t^{d-1},\,t^2-(\epsilon/2)t^d\hp\big).
\end{equation*} 
Note that $\sigma_\epsilon\big((2/\epsilon)^{\frac{1}{d-2}}\big)=\sigma_\epsilon(0)$, so $\sigma_\epsilon$ is not an embedding. Hence $\sigma$ is not an interior point of $\pd$ (with respect to both the spaces $\ad$ and $\bd$). This shows that {\it the space $\pd$, for $d\geq3$, is not open in both the spaces $\ad$ and $\bd$}.
\end{remark}

\begin{theorem}\label{th12}
The space $\qd$, for $d\geq2$, is dense in the space $\cd$.
\end{theorem}

\begin{proof}
Note that $\mathcal{Q}_2=\mathcal{C}_2$ and $\mathcal{Q}_3=\mathcal{C}_3$, so the theorem is trivially true for $d=2,3$. Assume that $d\geq4$. In order to prove that $\qd$ is dense in the space $\cd$, we have to show that for any $\phi\in\cd$ and any $\epsilon>0$, there is a polynomial knot $\psi\in\qd$ which belongs to the $\epsilon$-neighborhood $V(\phi,\epsilon)$ of $\phi$. Let an arbitrary $\phi\in\cd$ and an arbitrary $\epsilon>0$ be given. Let $\phi$ be given by $t\mapsto\fght$ and let
\begin{equation*}
m_1=\min\left\{\hs\lvert h'(t)\rvert : t\in\ro,\, g'(t)=0\hw\mbox{and}\hw h'(t)\neq0\hs\right\}.
\end{equation*} 
Let us choose a positive real number $r<\min\left\{m_1,\epsilon/2\right\}$, and let $\hat{h}:\ro\to\ro$ be defined by $t\mapsto h(t)+\frac{r}{2}t$. Suppose $g$ and $\hat{h}$ be given by $g(t)=\bt$ and $\hat{h}(t)=\ct$. Then for the curve $\alpha:\ro\to\rw$ given by $t\mapsto\big(g(t),\hp\hat{h}(t)\big)$, we have $\alpha'(t)\neq0$ for all $t\in\ro$, so by Lemma \ref{th7}, the self-intersections of $\alpha$ are exactly the common zeros of the following polynomials:
\begin{align*}
G(s, t)&=\pbst\quad\mbox{and}\\
H(s, t)&=\pcst.
\end{align*}
\hskip5.2mm We claim that the polynomials $G$ and $H$ do not have a non-constant common factor. Suppose on the contrary that they have a non-constant common factor, say $\Psi$, and let $\zeta$ be its leading term. Then $\zeta$ is a non-constant common factor of the polynomials $\mu(s,t)=s^{d-2}+s^{d-3}t+\cdots+t^{d-2}$ and $\nu(s,t)=s^{d-1}+s^{d-2}t+\cdots+t^{d-1}$. Let  $\mu_1$, $\nu_1$ and $\zeta_1$ be polynomials given by $\mu_1(t)=\mu(1,t)=1+t+\cdots+t^{d-2}$, $\nu_1(t)=\nu(1,t)=1+t+\cdots+t^{d-1}$ and $\zeta_1(t)=\zeta(1,t)$. It is obvious that the polynomials $\zeta$ and $\zeta_1$ have same degree, and thus $\zeta_1$ is a non-constant polynomial. Note that $t^{d-1}-1=\mu_1(t)(t-1)$ and $t^d-1=\nu_1(t)(t-1)$. Since $\zeta_1$ is a common factor of the polynomials $\mu_1$ and $\nu_1$, every root of $\zeta_1$ is a complex $(d-1)$\textsuperscript{th} root as well as $d$\hp\textsuperscript{th} root of unity other than $1$. This is a contradiction, since any complex $(d-1)$\textsuperscript{th} root of unity other than $1$ cannot be equal to a complex $d$\hp\textsuperscript{th} root of unity. This proves the claim.\vskip0.1mm 
By B\'{e}zout's theorem, the polynomials $G$ and $H$ have only finitely many common zeros, and hence the curve $\alpha$ has only finitely many self-intersections, say  $(s_1,t_1),(s_2,t_2),\ldots,(s_k,t_k)$ be those intersections (that is, $s_i\neq t_i$ and $\alpha(s_i)=\alpha(t_i)$ for $i=1,2,\ldots,k$). Let us choose a positive real number $u<\min\left\{m_2,\epsilon/2\right\}$, where
\begin{equation*}
m_2=\min\hf\left\{\hs\left|\frac{f(s_i)-f(t_i)}{s_i-t_i}\right|\,:\, i=1,2,\ldots,k\hw\mbox{and}\hw\frac{f(s_i)-f(t_i)}{s_i-t_i}\neq0\hs\right\}.
\end{equation*} 
Let $\hat{f}:\ro\to\ro$ be given by $t\mapsto f(t)+\frac{u}{2}t$. For $i=1,2,\ldots,k$, it is easy to note that $\hat{f}(s_i)\neq\hat{f}(t_i)$. This shows that the map $\psi:\ro\to\rt$ given by $t\mapsto\big(\hat{f}(t),\hp g(t),\hp\hat{h}(t)\big)$ has no self-intersections. Note that $\psi'(t)\neq0$ for all $t\in\ro$, since $\alpha'(t)\neq0$ for all $t\in\ro$. Also, it is easy to see that $\psi$ has a degree sequence $(d-2,\hp d-1,\hp d)$, so it is an element of the space $\qd$. We now estimate $\rho\hp(\phi,\psi)$ as follows:
\begin{align*}
\rho\hp\big(\phi,\psi\big)&\leq\rho\hp\big(\phi,(\hp f,\hp g,\hp\hat{h}\hp)\big)+\rho\hp\big((\hp f,\hp g,\hp\hat{h}\hp),\psi\big)\\
&\leq r/2+u/2\\
&\leq\epsilon/4+\epsilon/4\\
&\leq\epsilon/2\,.
\end{align*}
Thus, we have $\psi\in V(\phi,\epsilon)$. This proves the theorem.
\end{proof}

\begin{corollary}\label{th14}
The space $\qd$, for $d\geq2$, is dense in the space $\ad$.
\end{corollary}

\begin{corollary}\label{th18}
The spaces $\od$ and $\pd$, for $d\geq2$, are dense in the space $\ad$.
\end{corollary}

Our next goal is to study the path components of the spaces $\od$, $\pd$ and $\qd$. We first ensure that each one of them can have only finitely many path components by proving the following theorem and its corollaries. 
 
\begin{theorem}\label{th26}
The space $\od$, for $d\geq2$, is homeomorphic to a semialgebraic subset of $\ro^{3d}$ and hence it has only finitely many path components.
\end{theorem}

\begin{proof}
Let us denote real tuples $(a_0, a_1, \ldots, a_{d-2}),\ho(b_0, b_1, \ldots, b_{d-1})$\, and\, $(c_0, c_1, \ldots, c_d)$\, respectively by\, $\ba,\ho\bb$\, and\, $\bc$. For\, $\babc\in\ro^{3d}$, let $\phiabc:\ro\to\rt$ be a polynomial map given by 
\begin{equation*}
t\mapsto\abct\hp.
\end{equation*} 
Note that $\eta\hp(\phiabc)=\babc$ for $\babc\in\ro^{3d}$, where $\eta$ is the homeomorphism between the spaces $\ad$ and $\rtd$ as defined in Eq. (\ref{eq14}). Therefore, the  set $\eta\hp(\ad\setminus\od)$ can be written as follows:
\begin{equation*}
\eta\hp(\ad\setminus\od)=\left\{\hs\babc\in\ro^{3d}\mid\,\phiabc\hw\text{is not an embedding}\hs\right\}.
\end{equation*} 
For $\babc\in\ro^{3d}$, let $F_\ba,\hp G_\bb,\hp H_\bc:\rw\to\rt$ be polynomials which are given by
\begin{align*}
&(s,t)\xrightarrow{F_\ba}\past\hp,\\ 
&(s,t)\xrightarrow{G_\bb}\pbst\quad \mbox{and}\\  
&(s,t)\xrightarrow{H_\bc}\pcst\hp.
\end{align*}
For any $\babc\in\ro^{3d}$, by Lemma \ref{th8}, the polynomial map $\phiabc$ is not an embedding if and only if $F_\ba(s,t)=G_\bb(s,t)=H_\bc(s,t)=0$ for some $(s,t)\in\rw$. Therefore, the set $\eta\hp(\ad\setminus\od)$ is the image of the set
\begin{equation*}
U_d=\big\{\hs\abcst\in\ro^{3d+2}\mid F_\ba(s,t)=G_\bb(s,t)=H_\bc(s,t)=0\hs\big\}
\end{equation*} 
under the projection $\pi:\ro^{3d+2}\to\ro^{3d}$ onto the space of the first $3d$ coordinates. Note that the set $U_d$ is a semialgebraic subset of $\ro^{3d+2}$, so by Corollary \ref{th38}, the set $\eta\hp(\ad\setminus\od)$ is a semialgebraic subset of $\rtd$, hence so is its complement $\eta\hp(\od)=\ro^{3d}\setminus\eta\hp(\ad\setminus\od)$. By Theorem \ref{th39} and Proposition \ref{th40}, the set $\eta\hp(\od)$ has only finitely many path components and hence the same is true for the space $\od$ (since $\eta$ is a homeomorphism).
\end{proof}
	
\begin{corollary}\label{th27}
The space $\pd$, for $d\geq2$, is homeomorphic to a semialgebraic subset of $\ro^{3d}$ and hence it has only finitely many path components.
\end{corollary}

\begin{proof}
For $0\leq p<q<r\leq d$, let us consider the following set:
\begin{equation*}
W_{pqr}=\left\{\babc\in\ro^{3d}\mid\,\phiabc\hskip1.5mm\text{has the degree sequence}\hs (p,q,r)\right\}.
\end{equation*}
Note that for any $p<q<r\leq d$, the set $W_{pqr}$ is a semialgebraic subset of $\rtd$, hence so is the set $\eta\hp(\bd)=\mcup_{p<q<r}W_{pqr}$\hp. Therefore, the set $\eta\hp(\pd)=\eta\hp(\bd)\cap\eta\hp(\od)$ is a semialgebraic subset of $\rtd$ (since by Theorem \ref{th26}, the set $\eta\hp(\od)$ is semialgebraic). The rest half of the corollary follows from Theorem \ref{th39} and Proposition \ref{th40}.
\end{proof}
	
\begin{corollary}\label{th28}
The space $\qd$, for $d\geq2$, is homeomorphic to a semialgebraic subset of $\ro^{3d}$ and hence it has only finitely many path components.
\end{corollary}
	
\begin{proof}
Note that the sets $\eta\hp(\cd)$ and $\eta\hp(\od)$ are semialgebraic subsets of $\rtd$, hence so is their intersection $\eta\hp(\qd)=\eta\hp(\cd)\cap\eta\hp(\od)$. The rest half of the corollary follows from Theorem \ref{th39} and Proposition \ref{th40}.
\end{proof}

\begin{proposition}\label{th33}
Let $\phi\in\od$, and let $\alpha$ and $\beta$ be any real numbers such that $\alpha>0$. Suppose $\psi:\ro\rightarrow\rt$ be given by $t\mapsto\phi(\alpha\hp t +\beta)$. Then we have the following:
\begin{enumerate}[(1)]
\item Both $\phi$ and $\psi$ belong to the same path component of $\od$.
\item If $\phi\in\pd$, then both $\phi$ and $\psi$ belong to the same path component of $\pd$.
\item If $\phi\in\qd$, then both $\phi$ and $\psi$ belong to the same path component of $\qd$.
\end{enumerate}  
\end{proposition} 

\begin{proof}
Let $F:\zo\to\ad$ be defined by $s\mapsto F_s$, where $F_s:\ro\to\rt$ be given by
\begin{equation*}  
F_s(t)=\phi\big((1-s+\alpha s)t+\beta s\big)
\end{equation*} 
for $t\in\ro$. By Remark \ref{rm2}, for each $s\in\zo$, $F_s$ is an element of the space $\od$. Note that $F_0=\phi$ and $F_1=\psi$. This shows that $F$ is a path in $\od$ connecting $\phi$ and $\psi$.\vskip0.1mm
If $\phi$ belongs to $\pd$ (respectively, belongs to $\qd$), then for each $s\in\zo$, $F_s$ is an element of the space $\pd$ (respectively, of the space $\qd$). Thus, the map $F$ becomes a path in $\pd$ (respectively, in $\qd$) connecting $\phi$ and $\psi$. 
\end{proof}  

\begin{proposition}\label{th34}
Let $\phi:\ro\to\rt$ given by $t\mapsto\fght$ be an element of the space $\od$. For $i\in\{1,2,3\}$, let $\alpha_i$'s be positive real numbers, and let $\beta_i$'s and $\gamma_i$'s be any real numbers. Suppose $\psi:\ro\to\rt$ be given by $t\mapsto\big(\alpha_1 f(t)+\gamma_1,\hp \alpha_2\hp g(t)+\beta_1 f(t)+\gamma_2,\hp \alpha_3 h(t) +\beta_2 f(t)+\beta_3\hp g(t)+\gamma_3\big)$. Then we have the following:
\begin{enumerate}[(1)]
\item Both $\phi$ and $\psi$ belong to the same path component of $\od$.
\item If $\phi\in\pd$, then both $\phi$ and $\psi$ belong to the same path component of $\pd$.
\item If $\phi\in\qd$, then both $\phi$ and $\psi$ belong to the same path component of $\qd$.
\end{enumerate} 
\end{proposition} 

\begin{proof}  
For $s\in\zo$, define a map $T_s:\rt\to\rt$ by
\begin{align*}
T_s(x, y, z)&=\big( (1-s+\alpha_1s)x+\gamma_1s,\; \beta_1sx+(1-s+\alpha_2s)y+\gamma_2s,\; 
\beta_2sx+\beta_3sy\\ 
\hp&\hskip5.7mm+(1-s+\alpha_3s)z+\gamma_3s\big)
\end{align*}
for $(x, y, z)\in\rt$. Note that $T_s$, for $s\in\zo$, is an affine transformation. By Remark \ref{rm3}, for each $s\in\zo$, $H_s=T_s\circ\phi$ is an element of the space $\od$. Note that $H_0=\phi$ and $H_1=\psi$. Thus, the map $H:\zo\to\od$ given by $s\mapsto H_s$ is a path in $\od$ from $\phi$ to $\psi$.\vskip0.1mm
If $\phi$ belongs to $\pd$ (respectively, belongs to $\qd$), then for each $s\in\zo$, $H_s$ is an element of the space $\pd$ (respectively, of the space $\qd$). Therefore, the map $H$ becomes a path in $\pd$ (respectively, in $\qd$) connecting $\phi$ and $\psi$. 
\end{proof}  

\begin{corollary}\label{th35}
Every path component of each of the spaces $\od,\, \pd$ and $\qd$ is unbounded.
\end{corollary}  

\begin{proof}
Let $U$ be a path component of the space $\od$, and let $\phi$ be a polynomial knot belonging to it. In Proposition \ref{th34}, by taking $\alpha_i$'s arbitrarily large, and $\beta_i$'s and $\gamma_i$'s to be zero, the distance of $\psi$ from the zero map $0:\ro\to\rt$ (which is given by $0(t)=(0, 0, 0)$ for $t\in\ro$) can be made arbitrarily large. Thus, there are polynomial knots at arbitrarily large distances from the origin which belong to $U$. In other words, each path component of the space $\od$ is unbounded. Similarly, every path component of the spaces $\pd$ and $\qd$ is unbounded. 
\end{proof}

\begin{corollary}\label{th36}
Every neighborhood of the zero map $0\in\ad$ intersects each path component of the spaces $\od, \pd$ and $\qd$.
\end{corollary}

\begin{proof}
Let $U$ be a path component of the space $\od$, and let $\phi\in U$. In Proposition \ref{th34} by taking $\alpha_i$'s arbitrarily small, and $\beta_i$'s and $\gamma_i$'s to be zero, the distance of $\psi$ from the zero map can be made arbitrarily small. Hence, there are polynomial knots in $U$ which are arbitrarily close to the origin. In other wards, every neighborhood of the zero map intersects each path component of the space $\od$. Similarly, the result can be proved for the spaces $\pd$ and $\qd$. 
\end{proof}

\section{Main Results}\label{sec4}

In Section \ref{sec3}, we introduced the spaces $\od$, $\pd$ and $\qd$ consisting of polynomial knots of degree at most $d$. We define path equivalence in these spaces as:

\begin{definition}
Two polynomial knots in $\od$ (respectively, in $\pd$/in $\qd$) are said to be path equivalent in $\od$ (respectively, in $\pd$/in $\qd$) if they belong to its same path component.
\end{definition}

We would like to study the path equivalence of polynomial knots in the spaces $\od$, $\pd$ and $\qd$ as compared to their topological equivalence. We also estimate the path components of these spaces and look into the homotopy types of their path components. 

\subsection{Path components of the spaces}\label{sec4.1}

\begin{theorem}\label{th9}
Every polynomial knot in $\od$ is connected by a path in $\od$ to some polynomial knot in $\qd$.
\end{theorem} 
 
\begin{proof}
Let $\phi$ be an element of the space $\od$. By Corollaries \ref{th14} and \ref{th28}, the set $\eta\hp(\qd)$ is dense in $\rtd$ and it has only finitely many path components, so the element $\eta\hp(\phi)$ must belong to the closure of some path component, say $U$, of $\eta\hp(\qd)$. The set $\eta\hp(\qd)$ is a semialgebraic subset of $\rtd$, so by Theorem \ref{th39}, $U$ is also a semialgebraic subset of $\rtd$. This implies that the set $U\cup\{\eta\hp(\phi)\}$ is a semialgebraic subset of $\rtd$, since a one point set is always semialgebraic. Note that $U\cup\{\eta\hp(\phi)\}$ is connected; therefore, by Proposition \ref{th40}, it is path connected, and hence so is the set $\eta^{-1}(U)\cup\{\phi\}$. In other words, there is a path in $\eta^{-1}(U)\cup\{\phi\}\subseteq\od$ from $\phi$ to an element $\psi\in\eta^{-1}(U)\subseteq\qd$.
\end{proof}

\begin{corollary}\label{th21}
Every polynomial knot in $\od$ is connected by a path in $\od$ to some polynomial knot in $\pd$.
\end{corollary}

\begin{corollary}\label{th23}
Every polynomial knot in $\pd$ is connected by a path in $\pd$ to some polynomial knot in $\qd$.
\end{corollary}

For $\phi\in\ad$, let us denote its first, second and third components by $f_\phi,\, g_\phi$ and $h_\phi$ respectively, and for\, $i\in\{\hp0,1,\ldots, d\hp\}$,\, we denote the coefficients of $t^i$ in the polynomials $f_\phi,\, g_\phi$ and $h_\phi$ respectively by $a_{\phi i},\,b_{\phi i}$\, and $c_{\phi i}$. Sometimes we use letters like $\phi, \psi, \sigma, \tau$ and $\omega$ to denote the elements of $\ad$. In such cases, the components and their coefficients will be denoted using the respective subscripts.

\begin{theorem}\label{th4}
The space $\pd$, for $d\geq3$, is path connected. 
\end{theorem}

\begin{proof}
Let $\phi$ be an arbitrary but fixed element of the space $\pd$. By Corollary \ref{th23}, there is a path in $\pd$ from $\phi$ to some element $\psi\in\qd$. Note that $\qd$ is open in $\ad$, so we have a path in $\qd$ from $\psi$ to an element $\sigma\in\qd$ for which the coefficient $b_{\sigma1}$ of $t$ in the second component is nonzero. Let $\Lambda:\zo\to\qd$ be a map given by
\begin{equation*}
\Lambda(s)=\left(f_\sigma-a_{\sigma0}\hp s,\, g_\sigma-b_{\sigma0}\hp s,\, h_\sigma-\frac{c_{\sigma1}\hp g_\sigma-b_{\sigma0}\hp c_{\sigma1}+b_{\sigma1}\hp c_{\sigma0}}{b_{\sigma1}}\hp s\right)
\end{equation*} 
for $s\in\zo$. Let $\tau=\Lambda(1)$, then we have $a_{\tau0}=b_{\tau0}=c_{\tau0}=c_{\tau1}=0$ and $b_{\tau1}=b_{\sigma1}\neq0$. Obviously, the map $\Lambda$ is a path in $\qd$ from $\sigma$ to $\tau$.\vskip0.1mm 
Since $\qd$ is open in $\ad$, we can find a path in $\qd$ from $\tau$ to an element $\omega\in\qd$ which differs from $\tau$ only in the coefficient of $t^2$ in the third component and that coefficient is nonzero in $\omega$. For $s\in\zo$, we define polynomials $f_s,\,g_s$ and $h_s$ as follows:
\begin{align*}
f_s(t)&=a_{\omega1}s\,t+a_{\omega2}s^2\,t^2+\cdots+a_{\omega, d-2}s^{d-2}\,t^{d-2}\,,\\
g_s(t)&=b_{\omega1}t+b_{\omega2}s\,t^2+\cdots+b_{\omega, d-1}s^{d-2}\,t^{d-1}\quad\mbox{and}\\
h_s(t)&=c_{\omega2}t^2+c_{\omega3}s\,t^3+\cdots+c_{\omega d}s^{d-2}\,t^d.
\end{align*}
Let $\Gamma:\zo\to\ad$ be a map given by $\Gamma(s)=\Gamma_s$ for $s\in\zo$, where $\Gamma_s$ is defined as $t\mapsto\big(f_s(t),\hp g_s(t),\hp h_s(t)\big)$. Note that $a_{\omega0}=a_{\tau0}=0,\; b_{\omega0}=b_{\tau0}=0,\; c_{\omega0}=c_{\tau0}=0$ and $c_{\omega1}=c_{\tau1}=0$; therefore, we have $\Gamma(1)=(f_1, g_1, h_1)=\omega$. For $(s,t)\in\ozo\times\ro$, we have
\begin{align}
& f_s(t)=f_1(s\hf t),\; s\hf g_s(t)=g_1(s\hf t)\hw\mbox{and}\hw s^2\hf h_s(t)=h_1(s\hf t),\hskip1.5mm\mbox{and}\label{eq15}\\
& f_s'(t)=s\hp f_1'(s\hf t),\; g_s'(t)=g_1'(s\hf t)\hw\mbox{and}\hw s\hf h_s'(t)=h_1'(s\hf t).\label{eq27}
\end{align} 
\noindent Thus, for $s\in\ozo$ and $u,v,t\in\ro$ with $u\neq v$, we have $\Gamma_s'(t)=\big(f_s'(t),\,g_s'(t),\,h_s'(t)\big)\neq0$ and $\Gamma_s(u)\neq\Gamma_s(v)$ (since otherwise by Expressions (\ref{eq15}) and (\ref{eq27}) we will have $\omega'(s\hf t)=\big(f_1'(s\hf t),\,g_1'(s\hf t),\,h_1'(s\hf t)\big)=0$\, or\, $\omega(s\hf u)=\omega(s\hf v)$, a contradiction to the fact that $\omega$ is an embedding). This shows that the element $\Gamma(s)=\Gamma_s$, for any $s\in\ozo$, is an embedding and hence a member of the space $\pd$. Note that a map $\tilde{\omega}:\ro\to\rt$ given by $t\mapsto(0,\,b_{\omega1}t,\,c_{\omega2}t^2)$ is an element of the space $\pd$, since $b_{\omega1}=b_{\tau1}\neq0$ and $c_{\omega2}\neq0$. It is easy to see that $\Gamma(0)=(f_0, g_0, h_0)=\tilde{\omega}$. This shows that $\Gamma$ is a path in $\pd$ from $\tilde{\omega}$ to $\omega$.\vskip0.1mm 
Let $\phi_0$ be a fixed element of the space $\pd$ which is given by $t\mapsto(0,\,t,\,t^2)$. Take a map $\Omega:\zo\to\ad$ given by $s\mapsto\Omega_s$, where
\begin{equation*} 
\Omega_s(t)=\left(s(1-s)t,\,(b_{\omega1}\hf s+1-s)t+s(1-s)t^2,\,(c_{\omega2}\hf s+1-s)t^2+s(1-s)t^3\right)
\end{equation*}
for $t\in\ro$. Note that $\Omega$ is path in $\pd$ from $\phi_0$ to $\tilde{\omega}$. By using the paths (which are discussed above) between the consecutive elements in the list $\phi,\,\psi,\,\sigma,\,\tau,\,\omega,\,\tilde{\omega}$ and $\phi_0$, one gets a new path in $\pd$ from $\phi$ to the fixed element $\phi_0$ which belongs this space. 
\end{proof}

\begin{theorem}\label{th13}
If two polynomial knots are path equivalent in $\qd$, then they are topologically equivalent. 
\end{theorem}

\begin{proof}
Let $\phi$ and $\psi$ be two polynomial knots which belong to the same path component of the space $\qd$. Since $\qd$ is open in $\ad$, we can choose a smooth path $\alpha:\zo\to\qd$ such that $\alpha(0)=\phi$ and $\alpha(1)=\psi$. Consider the map $F:\zo\times\ro\to\rt$ given by $F(s,t)=\alpha_s(t)$ for $(s,t)\in\zo\times\ro$, where $\alpha_s=\alpha(s)$ for $s\in\zo$. This map is smooth and it is an isotopy of polynomial knots in $\qd$ connecting $\phi$ and $\psi$. Let $H:\zo\times\ro\to\zo\times\rt$ be the track of the isotopy $F$; that is, $H(s,t)=\big(s, F(s,t)\big)=\big(s,\alpha_s(t)\big)$ for $(s,t)\in\zo\times\ro$. Note that $H$ is a level preserving map and it is a smooth embedding of $\zo\times\ro$ in $\zo\times\rt$, so its image $H(\zo\times\ro)$ is a smooth submanifold of $\zo\times\rt$. One can identify $\zo\times\rt$ as a smooth submanifold of $\zo\times\st$ by the embedding $1\times\Gamma:\zo\times\rt\to\zo\times\st$, where $1$ is the identity mapping of $\zo$, and $\Gamma$ is the inverse of the stereographic projection from the north pole $N$ of $\st$. Let $K$ be the image of $(1\times\Gamma)\circ H$ and let\hskip1.7mm$\xbar{K}$\hskip1.7mm be its closure in $\zo\times\st$. It is obvious that the collection $\mathcal{S}=\big\{\zo\times\st\setminus\hf\xbar{K}, K, \zo\times\{N\}\big\}$ is a collection of pairwise disjoint smooth submanifolds of $\zo\times\st$ which covers it. Thus, the collection $\mathcal{S}$ becomes a pre-stratification for $\zo\times\st$ (see \cite[\S~5]{jm}). Note that $\mathcal{S}$ is a finite collection. Also, it is easy to check that it satisfies the condition of frontier (see \cite[\S~5]{jm}). Since $\zo\times\st\setminus\hf\xbar{K}$ is an open submanifold of $\zo\times\st$, its dimension is $4$. Therefore, it is trivial that the pairs $(\zo\times\st\setminus\hf\xbar{K}, K)$ and $(\zo\times\st\setminus\hf\xbar{K}, \zo\times\{N\})$ of strata in $\mathcal{S}$ satisfy the second Whitney condition (see \cite[\S~2]{jm}).\vskip0.1mm
 
We claim that the pair $(K, \zo\times\{N\})$ satisfies the second Whitney condition. Let $y=(u,N)$ be an element in $\zo\times\{N\}$, and let $\Phi:U\to\ro^4$ be a smooth chart of $\zo\times\st$ at $y$. We need to check that the pair $\big(\Phi(U\cap K), \Phi(U\cap\zo\times\{N\})\big)$ satisfies the second Whitney condition at $\Phi(y)$. Let $\{x_i\}\subseteq\Phi(U\cap K)$ and $\{y_i\}\subseteq\Phi(U\cap\zo\times\{N\})$ be sequences which converge to $\Phi(y)$, and they are such that: (a) $x_i\neq y_i$ for all $i\in\mathbb{N}$, (b) the tangent spaces $\tau_i=T\big(\Phi(U\cap K)\big)_{x_i}$ converge to a plane $\tau\subseteq\ro^4$ (the convergence is in the topology of Grassmannian of 2-planes in $\ro^4$), and (c) the secant lines $l_i=\overline{(x_i,y_i)}$ converge to a line $l\subseteq\ro ^4$. We can assume that $U=V\times W$ for some connected open neighborhoods $V$ of $u$ in $I$ and $W$ of $N$ in $\st$, and let $\Phi(s,x)=(s,\Psi(x))$ for $(s,x)\in V\times W$, where $\Psi$ is the restriction of the stereographic projection from the south pole of $\st$. For $i\in\mathbb{N}$, let 
\begin{equation*}
x_i=\Phi\circ(1\times\Gamma)\circ H(s_i,t_i)=\big(s_i, \Psi\circ\Gamma\circ F(s_i, t_i)\big)=\left(s_i,\,\left\| F(s_i, t_i)\right\|^{-2}F(s_i, t_i)\right)
\end{equation*} 
for some $(s_i, t_i)\in\zo\times\ro$, and let 
\begin{equation*}
y_i=\Phi(u_i, N)=\big(u_i, \Psi(N)\big)=(u_i, 0)
\end{equation*} 
for some $u_i\in\zo$. For $i\in\mathbb{N}$, note that the vector 
\begin{equation*}
\lambda_i=x_i-y_i=\left(s_i-u_i,\,\left\| F(s_i, t_i)\right\|^{-2}F(s_i, t_i)\right)
\end{equation*} 
spans the secant line $l_i$. Also, for $i\in\mathbb{N}$, it is easy to check that the vectors
\begin{align*}
v_i&=\left.\left(1\,,\;\frac{\partial}{\partial s}\Psi\circ\Gamma\circ F(s, t)\right)\right|_{(s_i,t_i)}\\[3pt]
&=\left.\left(1\,,\;\frac{\partial}{\partial s}\Bigg(\frac{F(s, t)}{\left\| F(s, t)\right\|^2}\Bigg)\right)\right|_{(s_i,t_i)}\\[3pt]
&=\left.\left(1\,,\;\frac{\frac{\partial}{\partial s}F(s, t)}{\left\| F(s, t)\right\|^2}-\frac{\frac{\partial}{\partial s}\left\| F(s, t)\right\|^2}{\left\| F(s, t)\right\|^4}F(s, t)\right)\right|_{(s_i,t_i)}\hskip7mm\mbox{and}\\[7pt]
w_i&=\left.\left(0\,,\;\frac{\partial}{\partial t}\Psi\circ\Gamma\circ F(s, t)\right)\right|_{(s_i,t_i)}\\[3pt]
&=\left.\left(0\,,\;\frac{\partial}{\partial t}\Bigg(\frac{F(s, t)}{\left\| F(s, t)\right\|^2}\Bigg)\right)\right|_{(s_i,t_i)}\\[3pt]
&=\left.\left(0\,,\;\frac{\frac{\partial}{\partial t}F(s, t)}{\left\| F(s, t)\right\|^2}-\frac{\frac{\partial}{\partial t}\left\| F(s, t)\right\|^2}{\left\| F(s, t)\right\|^4}F(s, t)\right)\right|_{(s_i,t_i)}
\end{align*}
span the tangent space $\tau_i$. After normalizing the vectors $\lambda_i$, $v_i$ and $w_i$, we get the unit vectors as follows:
\begin{align*}
\lambda_i'&=\frac{\lambda_i}{\left\|\lambda_i\right\|}\\[3pt]
&=\left(\frac{s_i-u_i}{\sqrt{(s_i-u_i)^2+\left\|F(s_i, t_i)\right\|^{-2}}}\,,\;\frac{F(s_i, t_i)}{\left\| F(s_i, t_i)\right\|^2\sqrt{(s_i-u_i)^2+\left\|F(s_i, t_i)\right\|^{-2}}}\right),\\[7pt]
v_i'&=\frac{v_i}{\left\|v_i\right\|}\\[3pt]
&=\left.\frac{\left\| F(s, t)\right\|^2}{\sqrt{\left\| F(s, t)\right\|^4+\left\|\frac{\partial}{\partial s}F(s, t)\right\|^2}}\left(1\,,\;\frac{\frac{\partial}{\partial s}F(s, t)}{\left\| F(s, t)\right\|^2}-\frac{\frac{\partial}{\partial s}\left\| F(s, t)\right\|^2}{\left\| F(s, t)\right\|^4}F(s, t)\right)\right|_{(s_i,t_i)}\\[1pt]
\noindent\mbox{and}\\[1pt]
w_i'&=\frac{w_i}{\left\|w_i\right\|}\\[3pt]
&=\left.\left(0\,,\;\frac{\frac{\partial}{\partial t}F(s, t)}{\left\|\frac{\partial}{\partial t}F(s, t)\right\|}-\frac{2\hp F(s,t)\cdot\frac{\partial}{\partial t}F(s, t)}{\left\| F(s, t)\right\|^2\left\|\frac{\partial}{\partial t}F(s, t)\right\|}F(s, t)\right)\right|_{(s_i,t_i)}\\[3pt]
&=\left.\left(0\,,\;\frac{\frac{\partial}{\partial t}F(s, t)}{\left\|\frac{\partial}{\partial t}F(s, t)\right\|}-\frac{2\cos\theta_i}{\left\| F(s, t)\right\|}F(s, t)\right)\right|_{(s_i, t_i)}\;,
\end{align*} 
where $\theta_i$ is the angle between the vectors $F(s_i,t_i)$ and $\frac{\partial F}{\partial t}(s_i,t_i)$. For $i\in\mathbb{N}$, the vector $\lambda_i'$ forms a basis for the secant line $l_i$, and the vectors $v_i'$ and $w_i'$ together form a basis for the tangent space $\tau_i$. Since $|t_i|\to\infty$ as $i\to\infty$, there exists a subsequence of the sequence $\{t_i\}$ which diverges to $\pm\infty$. We may assume that the sequence $\{t_i\}$ itself diverges to $\pm\infty$. Since the sequences $\left\{\|F(s_i,t_i)\|^{-1}F(s_i,t_i)\right\}$, $\left\{\|\frac{\partial F}{\partial t}(s_i, t_i)\|^{-1}\frac{\partial F}{\partial t}(s_i, t_i)\right\}$ and $\{\lambda_i'\}$ are bounded, they have convergent subsequences. Again, we may assume that these sequences themselves are convergent. It is easy to check that
\begin{equation}\label{eq4.1}
\lim_{i\to\infty}\left[\frac{\frac{\partial}{\partial s}F(s, t)}{\left\| F(s, t)\right\|^2}-\frac{\frac{\partial}{\partial s}{\left\| F(s, t)\right\|^2}}{\left\| F(s, t)\right\|^4}F(s, t)\right]_{(s_i, t_i)}=0.
\end{equation} 
Since the sequence $\{\lambda_i'\}$ is convergent, the sequence $\left\{\|\lambda_i\|^{-1}\|F(s_i, t_i)\|^{-1}\right\}$ (which is the sequence of norms of the second component of $\lambda_i'$) is also convergent. By Proposition \ref{th16}, the sequence $\{\theta_i\}$ converges to $0$ or $\pi$ depending on whether $\{t_i\}$ diverges to $\infty$ or $-\infty$; therefore, we have 
\begin{equation}\label{eq4.2}
\lim_{i\to\infty}\left[\frac{\frac{\partial}{\partial t}F(s, t)}{\left\|\frac{\partial}{\partial t}F(s, t)\right\|\cos\theta_i}-\frac{F(s, t)}{\left\| F(s, t)\right\|}\right]_{(s_i, t_i)}=0.
\end{equation}  
For $i\in\mathbb{N}$, let 
\begin{equation*}
z_i'=\frac{(s_i-u_i)\left\|v_i\right\|}{\left\|\lambda_i\right\|}v_i'-\frac{1}{\left\|\lambda_i\right\|\left\|F(s_i, t_i)\right\|\cos\theta_i}w_i'.
\end{equation*} 
Using Eqs. (\ref{eq4.1}) and (\ref{eq4.2}), we compute the limit of $z_i'$ as follows:
\begin{align*}
\lim_{i\to\infty}z_i'&=\lim_{i\to\infty}\left(\frac{(s_i-u_i)\left\|v_i\right\|}{\left\|\lambda_i\right\|}v_i'-\frac{1}{\left\|\lambda_i\right\|\left\|F(s_i, t_i)\right\|\cos\theta_i}w_i'\right)\\[3pt]
&=\lim_{i\to\infty}\left(\frac{s_i-u_i}{\left\|\lambda_i\right\|}\,,\;\frac{s_i-u_i}{\left\|\lambda_i\right\|}\left[\frac{\frac{\partial}{\partial s}F(s, t)}{\left\| F(s, t)\right\|^2}-\frac{\frac{\partial}{\partial s}{\left\| F(s, t)\right\|^2}}{\left\| F(s, t)\right\|^4}F(s, t)\right]_{(s_i, t_i)}\right.\\[2pt]
&\left.\hskip17mm-\frac{1}{\left\|\lambda_i\right\|\left\|F(s_i, t_i)\right\|}\left[\frac{\frac{\partial}{\partial t}F(s, t)}{\left\|\frac{\partial}{\partial t}F(s, t)\right\|\cos\theta_i}-\frac{2\hp F(s, t)}{\left\| F(s, t)\right\|}\right]_{(s_i, t_i)}\right)\\[3pt]
&=\lim_{i\to\infty}\left(\frac{s_i-u_i}{\left\|\lambda_i\right\|}\,,\;\frac{-1}{\left\|\lambda_i\right\|\left\|F(s_i, t_i)\right\|}\left(\frac{\frac{\partial}{\partial t}F(s, t)}{\left\|\frac{\partial}{\partial t}F(s, t)\right\|\cos\theta_i}-\frac{2\hp F(s, t)}{\left\| F(s, t)\right\|}\right]_{(s_i,t_i)}\right)\\[3pt]
&=\lim_{i\to\infty}\lambda_i'\;.
\end{align*}
This shows that the nonzero vector $\lambda'=\lim_{i\to\infty}\lambda_i'$ (which is same as $\lim_{i\to\infty}z_i'$) lies in both $l$ and $\tau$. Therefore, we have $l\subseteq\tau$. This shows that the pair $(K, \zo\times\{N\})$ satisfies the second Whitney condition.\vskip0.1mm
Now the collection $\mathcal{S}$ becomes a Whitney pre-stratification for $\zo\times\st$ (see \cite[\S~5]{jm}). Let $\pi:\zo\times\st\to\zo$ be the projection onto the first coordinate, then this map is proper. It is easy to check that the map $\pi$, when restricted to any stratum in $\mathcal{S}$ is a submersion. Thus, by Thom's First Isotopy Lemma (see \cite[\S~11]{jm}), the map $\pi$, when restricted to any stratum in $\mathcal{S}$ is a locally trivial fibration and hence a trivial bundle over $I$. In particular $\pi\mid_{\zo\times\st\setminus\bar{K}}:\zo\times\st\setminus\hf\xbar{K}\to\zo$ is a trivial bundle and hence the fibers $\pi^{-1}(0)=\{0\}\times\st\setminus\tilde{\alpha}_0(\so)$ and $\pi^{-1}(1)=\{1\}\times\st\setminus\tilde{\alpha}_1(\so)$ are homeomorphic, where $\tilde{\alpha}_0:\so\to\st$ and $\tilde{\alpha}_1:\so\to\st$ are extensions of $\alpha_0$ and $\alpha_1$ respectively. Since it is known that the complements of the knots in $\st$ are homeomorphic if and only if they are ambient isotopic (see \cite{gl}), the knots $\tilde{\alpha}_0$ and $\tilde{\alpha}_1$ are ambient isotopic. In other words, the knots $\phi=\alpha_0$ and $\psi=\alpha_1$ are topologically equivalent.
\end{proof} 

\begin{remark}\label{rm4}
Note that the converse of Theorem \ref{th13} is false. More precisely, a polynomial knot in $\qd$ given by $t\mapsto\fght$ is topologically equivalent to the polynomial knot $t\mapsto\big(f(t),\hp -g(t),\hp -h(t)\big)$, but they belong to the distinct path components of the space $\qd$.
\end{remark}

\begin{remark}\label{rm5}
For $d\geq3$ and a polynomial knot $t\mapsto\fght$ in $\qd$, there are eight distinct path components of $\qd$ each of which contains exactly one of the knot $t\mapsto\big((e_1 f(t),e_2\hp g(t),e_3\hp h(t)\big)$ for $(e_1,e_2,e_3)$ in $\{-1,1\}^3$. Therefore, the total number of path components of the space $\qd$, for $d\geq3$, are in multiples of eight.
\end{remark}

\begin{remark}\label{rm6}
For $d\geq3$, if there are $n$ distinct knots (up to ambient isotopy and mirror images) which can be represented in $\qd$, then it has at least $8n$ distinct path components.
\end{remark}

It is easy to note that the space $\mathcal{Q}_2$ has exactly four path components. The spaces $\mathcal{Q}_3$ and $\mathcal{Q}_4$ each has exactly eight path components (see \cite[Proposition 3.14 and Theorem 3.22]{mr}). Furthermore, the spaces $\mathcal{Q}_5$, $\mathcal{Q}_6$ and $\mathcal{Q}_7$ have at least $16$, $24$ and $88$ path components respectively (see \cite[\S 3]{mr}).\vskip0.1mm

It is known that if two polynomial knots are topologically equivalent, then they are isotopic by an isotopy of polynomial knots (see \cite{rs1}). In the following theorem, we prove that the converse of this is false. 

\begin{theorem}\label{th15}
Every polynomial knot is isotopic to some trivial polynomial knot by a smooth isotopy of polynomial knots.
\end{theorem}
\begin{proof}
Let a polynomial knot  $\phi:\ro\to\rt$ be given, and let it be given by
\begin{equation*}
\phi(t)=\left(a_0+a_1t+\cdots+a_dt^d,\,b_0+b_1t+\cdots+b_dt^d,\,c_0+c_1t+\cdots+c_dt^d\right)
\end{equation*} 
for $t\in\ro$, where $d$ is the degree of $\phi$. For any but fixed $\epsilon>0$ and for $s\in(-\epsilon, 1+\epsilon)$, define a map $F_s:\ro\to\rt$ by 
\begin{equation*}
F_s(t)=(a_0, b_0, c_0)\hp s+(a_1, b_1, c_1)\hp t+\sum_{i=2}^d\ho(a_i, b_i, c_i)\hp s^{i-1}\hp t^i
\end{equation*} 
for $t\in\ro$. Let $F:(-\epsilon,1+\epsilon)\times\ro\to\rt$ be a map given by $F(s,t)=F_s(t)$ for $(s,t)\in(-\epsilon,1+\epsilon)\times\ro$. Note that $\phi$ is a smooth embedding, so for $s\in(-\epsilon,1+\epsilon)\setminus\{0\}$ and $t,u,v\in\ro$ with $u\neq v$, we have
\begin{equation*}
F_s'(t)=\phi'(s\hp t)\neq0 \quad\mbox{and}\quad F_s(u)-F_s(v)=\frac{\phi(s \hp u)-\phi(s\hp v)}{s}\neq0\hf.
\end{equation*} 
This shows that $F_s$ is a smooth embedding (a polynomial knot) for all $s\in(-\epsilon,1+\epsilon)\setminus\{0\}$. Let $\psi:\ro\to\rt$ be a map given by $t\mapsto(a_1t,\,b_1t,\,c_1t)$. Since at least one of $a_1, b_1$ or $c_1$ is nonzero (otherwise $\phi'(0)=0$), the map $\psi$ is a smooth embedding (a polynomial knot). Note that $F(0,t)=\psi(t)$ and $F(1,t)=\phi(t)$ for all $t\in\ro$. Also, it is easy to see that $F$ is smooth, since its components are polynomials in two variables. Thus, the map $F$ is a smooth isotopy of polynomial knots connecting $\phi$ and $\psi$. 
\end{proof}

\begin{corollary}\label{th22}
Every polynomial knot is connected to a trivial polynomial knot by a smooth path in the space $\spk$ of all polynomial knots.
\end{corollary}

\begin{proof}
For a given polynomial knot $\phi$, a map $\alpha:\zo\to\spk$ given by $s\mapsto F_s$ (where $F_s$, for $s\in\zo$, is defined as in the proof of Theorem \ref{th15}) is a smooth path in $\spk$ connecting $\phi$ to a trivial polynomial knot.
\end{proof}

\subsection{Homotopy types of the spaces}\label{sec4.2}

\begin{theorem}\label{th2}
The space $\mathcal{O}_2$ has the same homotopy type as $\so$.
\end{theorem}

\begin{proof}
For $\phi\in\mathcal{O}_2$, since $\phi'(0)\neq0$, we must have $(b_{\phi1})^2+(c_{\phi1})^2\neq0$. Define maps $\Psi:\mathcal{O}_2\to\so$ and $\Omega:\so\to\mathcal{O}_2$ by
\begin{equation*}
\Psi(\phi)=\frac{1}{\sqrt{(b_{\phi1})^2+(c_{\phi1})^2}}\left(b_{\phi1},\hp c_{\phi1}\right)
\end{equation*}
for $\phi\in\mathcal{O}_2$ and
$\Omega(x)=\phi_x$ for $x=(x_1, x_2)\in\so$, where $\phi_x(t)=(0,x_1\hp t,x_2\hp t)$ for $t\in\ro$. The maps $\Psi$ and $\Omega$ are continuous. Let $F:\zo\times\mathcal{O}_2\to\mathcal{A}_2$ be a map given by $F(s,\phi)=F_{s\phi}$ for $(s,\phi)\in\zo\times\mathcal{O}_2$, where 
\begin{equation*}
F_{s\phi}(t)=(1-s)\phi(t)+\frac{s}{\sqrt{(b_{\phi1})^2+(c_{\phi1})^2}}\left(0,\,b_{\phi1}t,\,c_{\phi1}t\right)
\end{equation*}
for $t\in\ro$. Note that the map $F$ is continuous. The space $\mathcal{O}_2$ is the disjoint union of the following sets:
\begin{align*}
\mathcal{O}_{21}&=\big\{\phi\in\mathcal{A}_2\mid\phi\hskip1.5mm\mbox{has degree sequence}\hskip1.5mm(0,0,1)\big\}\hskip2.5mm\mbox{and}\\
\mathcal{O}_{22}&=\big\{\phi\in\mathcal{A}_2\mid\ho\mbox{second component of}\hskip1.5mm\phi\hskip1.5mm\mbox{is linear}\big\}.
\end{align*} 
One can check that $F(s,\phi)\in\mathcal{O}_{21}$ for all $(s,\phi)\in\zo\times\mathcal{O}_{21}$ and $F(s,\psi)\in\mathcal{O}_{22}$ for all $(s,\psi)\in\zo\times\mathcal{O}_{22}$. This shows that $F$ maps $\zo\times\mathcal{O}_2$ into the space $\mathcal{O}_2$. Note that $F(0,\phi)=\phi$ and $F(1,\phi)=\Omega\big(\Psi(\phi)\big)$ for all $\phi\in\mathcal{O}_2$. This shows that $\Omega\circ\Psi$ is homotopic to the identity map of $\mathcal{O}_2$. Also, it is easy to see that the map $\Psi\circ\Omega$ is the identity map of $\so$.
\end{proof}

\begin{theorem}\label{th3}
The space $\od$, for $d\geq3$, has the same homotopy type as $\sw$.
\end{theorem}

\begin{proof}
For $\phi\in\od$, since $\phi'(0)=0$, we have $(a_{\phi1})^2+(b_{\phi1})^2+(c_{\phi1})^2\neq0$. Let $\Gamma_d:\mathcal{O}_d\to\sw$ and $\Upsilon_d:\sw\to\od$ be maps given by
\begin{equation*}
\Gamma_d(\phi)=\frac{1}{\sqrt{(a_{\phi1})^2+(b_{\phi1})^2+(c_{\phi1})^2}}\left(a_{\phi1}\hp, b_{\phi1}\hp, c_{\phi1}\right)
\end{equation*}
for $\phi\in\od$ and $\Upsilon_d(x)=\phi_x$ for $x=(x_1,x_2,x_3)\in\sw$, where $\phi_x(t)=(x_1\hp t, x_2\hp t,x_3\hp t)$ for $t\in\ro$. The maps $\Gamma_d$ and $\Upsilon_d$ are continuous. Let $\mu:\zo\to\rt$ be a map given by 
\begin{equation*} 
\mu(s)= s+\frac{1-s}{\sqrt{(a_{\phi1})^2+(b_{\phi1})^2+(c_{\phi1})^2}}\,
\end{equation*}  
for $s\in\zo$. Let $H_d:\zo\times\od\to\ad$ be a map given by $H_d(s,\phi)=H_{s\phi}$ for $(s,\phi)\in\zo\times\od$, where
\begin{equation*} 
H_{s\phi}(t)=(a_{\phi0}, b_{\phi0}, c_{\phi0})\hp s+\mu(s)\left((a_{\phi1}, b_{\phi1}, c_{\phi1})\hp t+\sum_{i=2}^d\ho(a_{\phi i}, b_{\phi i}, c_{\phi i})\hp s^{i-1}\hp t^i\right)
\end{equation*}
for $t\in\ro$. Note that $H_d$ is a continuous map, and $H_d(0,\phi)=\Upsilon_d\big(\Gamma_d(\phi)\big)$ and $H_d(1,\phi)=\phi$ for all $\phi\in\od$. For $s\in\ozo$, $\phi\in\od$ and $u,v,t\in\ro$ with $u\neq v$, we have $H_{s\phi}'(t)=\mu(s)\hp\phi'(s\hp t)\neq0$ and
\begin{equation*}
H_{s\phi}(u)-H_{s\phi}(v)=\frac{\mu(s)\big(\phi(s \hp u)-\phi(s\hp v)\big)}{s}\neq0.
\end{equation*}
This shows that the map $H_{s\phi}$ is a smooth embedding for all $(s,\phi)\in\zo\times\od$, and hence the image of $H_d$ is contained in $\od$.  This proves that $\Upsilon_d\circ\Gamma_d$ is homotopic to the identity map of $\od$. Also, note that the map $\Gamma_d\circ\Upsilon_d$ is the identity map of $\sw$.
\end{proof}

\begin{corollary}\label{th19}
The space $\od$, for $d\geq2$, is path connected.
\end{corollary}

\begin{corollary}\label{th20}
The space $\spk$ of all polynomial knots has the same homotopy type as $\sw$.
\end{corollary}

\begin{proof}
Since $\spk=\mcup_{d\geq3}\od\,$, we define maps $\Gamma:\spk\to\sw$,\, $\Upsilon:\sw\to\spk$ and $H:\zo\times\spk\to\spk$ given by $\Gamma(\phi)=\Gamma_d(\phi)$ for $\phi\in\od$,  
$\Upsilon(x)=\Upsilon_3(x)$ for $x\in\sw$ and $H(s,\psi)=H_d(s,\psi)$ for $(s,\psi)\in\zo\times\od$ (where the maps $\Gamma_d,\, \Upsilon_3$ and $H_d$, for  $d\geq3$, are defined as in the proof of Theorem \ref{th3}). It is easy to check that the maps $\Gamma$ and $H$ are well defined; that is, for $n>d\geq 3$, $\Gamma_n(\phi)=\Gamma_d(\phi)$ for all $\phi\in\od$ and $H_n(s,\psi)=H_d(s,\psi)$ for all $(s,\psi)\in\zo\times\od$.\vskip0.1mm 
To prove the continuity of the maps $\Gamma,\, \Upsilon$ and $H$, let us consider an open set $U$ in $\sw$ and an open set $V$ in $\spk$. It is easy to check that
\begin{align}
&\Gamma^{-1}(U)\cap\od=\Gamma_d^{-1}(U),\label{eq4.5}\\
&\Upsilon^{-1}(V)=\Upsilon_3^{-1}(V\cap\mathcal{O}_3)\quad\mbox{and}\label{eq4.6}\\ &H^{-1}(V)\cap\zo\times\od=H_d^{-1}(V\cap\od)\label{eq4.7}
\end{align} 
for all $d\geq3$. Note that $\Gamma_d^{-1}(U)$ is open in $\od$ for all $d\geq3$ and $\Upsilon_3^{-1}(V\cap\mathcal{O}_3)$ is open in $\sw$, so by Eqs. \ref{eq4.5} and \ref{eq4.6}, $\Gamma^{-1}(U)$ is open in $\spk$ and $\Upsilon^{-1}(V)$ is open in $\sw$. This proves the continuity of the maps $\Gamma$ and $\Upsilon$. Also, note that $H_d^{-1}(V\cap\od)$ is open in $\zo\times\od$ for all $d\geq3$, so by Eq. \ref{eq4.7}, the set $H^{-1}(V)$ is open in $\zo\times\spk$ with respect to the inductive limit topology which comes from the stratification $\zo\times\spk=\mcup_{d\geq3}\hf\zo\times\od$. Since $\zo$ is compact and regular, the inductive limit typology on $\zo\times\spk$ is same as the product topology on it (see \cite[\S 18.5]{jhcw} and \cite[\S 8]{imj}). Thus, the set $H^{-1}(V)$ is open in $\zo\times\spk$ with respect to its product topology. This proves that the map $H$ is continuous.\vskip0.1mm
Since $H_d(0,\phi)=\Upsilon_3\big(\Gamma_d(\phi)\big)$ and $H_d(1,\phi)=\phi$ for all $d\geq3$ and for all $\phi\in\od$, we have $H(0,\psi)=\Upsilon\big(\Gamma(\psi)\big)$ and $H(1,\psi)=\psi$ for all $\psi\in\spk$. This shows that the map $\Upsilon\circ\Gamma$ is homotopic to the identity map of $\spk$. Also, it is easy to note that the map $\Gamma\circ\Upsilon$  which is same as $\Gamma_3\circ\Upsilon_3$ is the identity map of $\sw$.
\end{proof}

\begin{corollary}\label{th30}
The space $\spk$ is path connected.
\end{corollary}

We have noted in Theorem \ref{th13} that the space $\qd,$ for $d\geq 2,$ is not path connected. It is easy to see that the path components of the spaces $\mathcal{Q}_2$ and $\mathcal{Q}_3$ are contractible. In general it becomes difficult to make any inference about the path components of $\qd$ for higher values of $d$. Here, we discuss the homotopy types of the path components of the spaces $\mathcal{Q}_4$ and $\mathcal{Q}_5$. First we prove that every path component of the space $\mathcal{Q}_4$ is contractible by proving that a one point subspace is a {\it weak deformation retract} in the sense of the following definition. 

\begin{definition}
We say that a subset $A$ of a topological space $X$ is a weak deformation retract if there is a continuous map $F:\zo\times X\to X$ such that $F(0,x)=x$ and $F(1,x)\in A$ for all $x\in X$. The map $F$ is called a weak deformation retraction.
\end{definition}

\begin{remark}
A topological space $X$ is contractible if and only if a one point subspace is a weak deformation retract of it. Furthermore, a weak deformation retract of a topological space is homotopy equivalent to it.
\end{remark}

\begin{lemma}\label{th5.6}
Let $A$ be a weak deformation retract of a space $X$. Then a weak deformation retract of $A$ is also a weak deformation retract of $X$. 
\end{lemma}

\begin{remark}\label{rm5.1}
Lemma \ref{th5.6} says that, being a weak deformation retract is a transitive relation. In particular, if a contractible subspace is a weak deformation retract of a space $X$, then the space $X$ is contractible.
\end{remark}

\begin{theorem}\label{th5.13}
The path components of the space $\mathcal{Q}_4$ are contractible.
\end{theorem}

\begin{proof}
By Theorem 3.22 in \cite{mr}, the space $\mathcal{Q}_4$ has exactly eight path components which are given by
\begin{equation*}
\mathcal{Q}_{4e}=\big\{\hs \phi\in\mathcal{Q}_4\mid\,e_1a_{\phi2}>0,\, e_2b_{\phi3}>0\ho\mbox{and}\ho e_3c_{\phi4}>0\hs\big\}\hp,
\end{equation*} 
where $e=(e_1,e_2,e_3)\in\cmoo^3$. In order to prove the theorem, we have to prove that the space $\mathcal{Q}_{4e}$, for $e\in\cmoo^3$, is contractible. It can be easily seen that the set 
$$\mathcal{L}_{4e}=\big\{\hs\phi\in\mathcal{Q}_{4e}\mid a_{\phi0}=b_{\phi0}=c_{\phi0}=0\hs\big\}$$
is a weak deformation retract of the space $\mathcal{Q}_{4e}$. For $\phi\in\mathcal{L}_{4e}$, let $f_\phi, g_\phi$ and $h_\phi$ be the first, second and third components of $\phi$ respectively, and let $F:\zo\times\mathcal{L}_{4e}\to\mathcal{L}_{4e}$ be a map given by $(s,\phi)\mapsto F_{s\phi}$, where
{\small\begin{equation*}
F_{s\phi}(t)=\left(f_\phi(t),\hs g_\phi(t)-\frac{b_{\phi2}}{a_{\phi2}}s\hf f_\phi(t),\hs h_\phi(t)+\frac{b_{\phi2}c_{\phi3}-b_{\phi3}c_{\phi2}}{a_{\phi2}b_{\phi3}}s\hf f_\phi(t)-\frac{c_{\phi3}}{b_{\phi3}}s\hf g_\phi(t)\right)
\end{equation*}}\vskip0.1mm
\noindent for $t\in\ro$. The map $F$ is continuous. Let us consider a set
\begin{equation*}
\mathcal{M}_{4e}=\left\{\hs\phi\in\mathcal{Q}_{4e}\mid t\xrightarrow{\phi}(\hp a_{\phi2}t^2+a_{\phi1}t,\hp b_{\phi3}t^3+b_{\phi1}t,\hp c_{\phi4}t^4+c_{\phi1}t\hp)\hs\right\}\hp.
\end{equation*}
\noindent One can see that $F(0,\phi)=\phi$ and $F(1,\phi)\in\mathcal{M}_{4e}$ for all $\phi\in\mathcal{L}_{4e}$. Thus, the space $\mathcal{M}_{4e}$ is a weak deformation retract of the space $\mathcal{L}_{4e}$. Let $H:\zo\times\mathcal{M}_{4e}\to\mathcal{M}_{4e}$ be a map which is given by $(s,\phi)\mapsto H_{s\phi}$, where
{\small\begin{equation*}
H_{s\phi}(t)=\left(\hs\bigg( 1-s+\frac{1}{\lvert a_{\phi2}\rvert}s\bigg)f_\phi(t),\hs \bigg( 1-s+\frac{1}{\lvert b_{\phi3}\rvert}s\bigg) g_\phi(t),\hs\bigg( 1-s+\frac{1}{\lvert c_{\phi4}\rvert}s\bigg) h_\phi(t)\hs\right)
\end{equation*}}\vskip0.1mm
\noindent for $t\in\ro$. Note that the map $H$ is continuous. Consider the following set:
\begin{equation*}
\mathcal{N}_{4e}=\big\{\hs \phi\in\mathcal{Q}_4\mid t\xrightarrow{\phi}(\hp e_1\hf t^2+a_\phi t,\hp e_2\hf t^3+b_\phi t,\hp e_3\hf t^4+c_\phi t\hp)\hs\big\}.
\end{equation*}
\noindent It is clear that $H(0,\phi)=\phi$ and $H(1,\phi)\in\mathcal{N}_{4e}$ for all $\phi\in\mathcal{M}_{4e}$. Therefore, $\mathcal{N}_{4e}$ is a weak deformation retract of the space $\mathcal{M}_{4e}$. Let us define a map $\Gamma:\zo\times\mathcal{N}_{4e}\to\mathcal{N}_{4e}$ by $(s,\phi)\mapsto\Gamma_{s\phi}$, where 
{\small\begin{equation*}
\Gamma_{s\phi}(t)=\Big(e_1t^2+a_\phi\hp t,e_2t^3+(b_\phi+e_2s\left| b_\phi\right| +e_2s)\hf t, e_3t^4+(c_\phi-2e_1e_3s a_\phi\left| b_\phi\right|-2e_1e_3s a_\phi)\hf t\Big)
\end{equation*}}\vskip0.1mm
\noindent for $t\in\ro$. The map $\Gamma$ is continuous. By Corollary 3.16 (see also Proposition 3.21) in \cite{mr}, the space $\mathcal{N}_{4e}$ is a union of the following sets: \begin{align*}
\mathcal{N}_{4e}^1&=\big\{\hf\phi\in\mathcal{N}_{4e}\mid 3a_\phi^2+4e_2b_\phi>0\hf\big\}\hf,\\
\mathcal{N}_{4e}^2&=\big\{\hf\phi\in\mathcal{N}_{4e}\mid e_1a_\phi^3+2e_1e_2a_\phi b_\phi+e_3c_\phi>0\hf\big\}\hw\mbox{and}\\
\mathcal{N}_{4e}^3&=\big\{\hf\phi\in\mathcal{N}_{4e}\mid e_1a_\phi^3+2e_1e_2a_\phi b_\phi+e_3c_\phi<0\hf\big\}\hf.
\end{align*}\noindent One can check that: (a) $\Gamma_{s\phi}\in\mathcal{N}_{4e}^1$ for all $(s,\phi)\in\ozo\times\mathcal{N}_{4e}^1$; (b) $\Gamma_{s\phi}\in\mathcal{N}_{4e}^2$ for all $(s,\phi)\in\ozo\times\mathcal{N}_{4e}^2$; (c) $\Gamma_{s\phi}\in\mathcal{N}_{4e}^3$ for all $(s,\phi)\in\ozo\times\mathcal{N}_{4e}^3$; (d) $\Gamma_{0\phi}=\phi$ and $\Gamma_{1\phi}\in\mathcal{N}_{4e}^1$ for all $\phi\in\mathcal{N}_{4e}$. Thus, the image of $\Gamma$ is contained in $\mathcal{N}_{4e}$, and we have $\Gamma(0,\phi)=\phi$ and $\Gamma(1,\phi)\in\mathcal{N}_{4e}^1$ for all $\phi\in\mathcal{N}_{4e}$. This shows that the space $\mathcal{N}_{4e}^1$ is a weak deformation retract of the space $\mathcal{N}_{4e}$. Let $\Lambda:\zo\times\mathcal{N}_{4e}^1\to\mathcal{N}_{4e}^1$ be a map given by $(s,\phi)\mapsto\Lambda_{s\phi}$, where 
\begin{equation*}
\Lambda_{s\phi}(t)=\big(\hp e_1t^2+(a_\phi-a_\phi s)\hp t, e_2t^3+(b_\phi-2b_\phi s+b_\phi s^2+e_2s)\hp t, e_3t^4+(c_\phi-c_\phi s)\hp t\hp\big)
\end{equation*} 
\noindent for $t\in\ro$. Note that the map $\Lambda$ is continuous. Let $\phi_0$ be a fixed element in $\mathcal{Q}_{4e}$ which is given by $t\mapsto(\hp e_1t^2, e_2t^3+e_2t, e_3t^4\hp)$. Note that $\Lambda(0,\phi)=\phi$ and $\Lambda(1,\phi)=\phi_0$ for all $\phi\in\mathcal{N}_{4e}^1$. Therefore, the space $\mathcal{N}_{4e}^1$ is contractible, and hence by Lemma \ref{th5.6} and Remark \ref{rm5.1}, the space $\mathcal{Q}_{4e}$ is also contractible.
\end{proof}

For $e=(e_1,e_2,e_3)$ in $\cmoo^3$, let
\begin{equation*}
\mathcal{Q}_{5e}=\big\{\hs \phi\in\mathcal{Q}_5\mid\,e_1a_{\phi3}>0,\, e_2b_{\phi4}>0\ho\mbox{and}\ho e_3c_{\phi5}>0\hs\big\}\hp.
\end{equation*} 
\noindent Note that the space $\mathcal{Q}_5$ is disjoint union of the spaces $\mathcal{Q}_{5e}$ for $e\in\cmoo^3$.

\begin{proposition}\label{th5.14}
For $e=(e_1,e_2,e_3)\in\cmoo^3$, let $p_e=e_1e_2e_3$. Then we have the following:
\begin{enumerate}[(1)]
\item If $p_e=1$, then the space $\mathcal{Q}_{5e}$ contains only one path component corresponding to the trefoil knot. Moreover, this path component is contractible. The space $\mathcal{Q}_{5e}$ does not contain any polynomial knot which represent the mirror image of the trefoil.
\item If $p_e=-1$, then the space $\mathcal{Q}_{5e}$ contains only one path component corresponding to the mirror image of the trefoil, and this path component is contractible. The space $\mathcal{Q}_{5e}$ does not contain any polynomial knot which represent the trefoil knot.
\end{enumerate}
\end{proposition}

\begin{proof}
Let $\mathcal{K}_{3,4,5}$ be the space of all polynomial knots of the type
\begin{equation*}
t\mapsto\left(t^3+a_2t^2+a_1t,\hs t^4+b_3t^3+b_2t^2+b_1t,\hs t^5+c_4t^4+c_3t^3+c_2t^2+c_1t\right)\hp.
\end{equation*} 
\noindent By Proposition 4-9 in \cite{sky}, we have the following: 
\begin{enumerate}[(a)]
\item The space $\mathcal{K}_{3,4,5}$  contains only one path component which corresponds to the trefoil knot. Moreover, this path component is contractible.
\item The space $\mathcal{K}_{3,4,5}$ does not contain any polynomial knot which represent the mirror image of the trefoil.
\end{enumerate}
\noindent For $e\in\cmoo^3$, let $F:\mathcal{Q}_{5e}\to\mathcal{K}_{3,4,5}$ and $H:\mathcal{K}_{3,4,5}\to\mathcal{Q}_{5e}$ be the maps given by
\begin{equation*}
F(\phi)=\left(\dfrac{f_\phi-a_{\phi0}}{a_{\phi3}},\hs \frac{g_\phi-b_{\phi0}}{b_{\phi4}},\hs \frac{h_\phi-c_{\phi0}}{b_{\phi5}}\right)\quad\text{and}\quad 
H(\psi)=\left(e_1f_\psi,\hs e_2g_\psi,\hs e_3h_\psi\right)\hp,
\end{equation*}
\noindent for $\phi\in\mathcal{Q}_{5e}$ and $\psi\in\mathcal{K}_{3,4,5}$. It is easy to note the following statements:
\begin{enumerate}[(a)]\setcounter{enumi}{2}
\item If $p_e=1$, then a polynomial knot $\phi\in\mathcal{Q}_{5e}$ represents the trefoil knot if and only if the polynomial knot $F(\phi)\in\mathcal{K}_{3,4,5}$ represents the same knot. The space $\mathcal{Q}_{5e}$ does not contain any polynomial knot which represent the mirror image of the trefoil (since by the statement (b) above). 
\item If $p_e=-1$, then a polynomial knot $\phi\in\mathcal{Q}_{5e}$ represents the mirror image the trefoil if and only if the polynomial knot $F(\phi)\in\mathcal{K}_{3,4,5}$ represents the trefoil knot. The space $\mathcal{Q}_{5e}$ does not contain any polynomial knot which represent the trefoil knot (since by the statement (b) above).
\end{enumerate}
\noindent Note that the map $F\circ H$ is the identity map of $\mathcal{K}_{3,4,5}$. Also, it can be easily checked that the map $H\circ F$ is homotopy equivalent to the identity map of the space $\mathcal{Q}_{5e}$. Thus, we have the following: 
\begin{enumerate}[(a)]\setcounter{enumi}{4}
\item The space $\mathcal{Q}_{5e}$ is homotopy equivalent to the space $\mathcal{K}_{3,4,5}$.
\end{enumerate}
\noindent The statements (a), (c) and (e) together imply the statement (1) of the proposition, and the statements (a), (d) and (e) together imply the statement (2) of the proposition.
\end{proof}

\begin{proposition}\label{th5.15}
The space $\mathcal{Q}_5$ has exactly eight path components corresponding to the trefoil knot and its mirror image. Moreover, all these path components are contractible.  
\end{proposition}

\begin{proof}
Since the space $\mathcal{Q}_5$ is disjoint union of the spaces $\mathcal{Q}_{5e}$, for $e\in\cmoo^3$, the result follows immediately from Proposition \ref{th5.14}. 
\end{proof}

By Remark \ref{rm5}, the space $\mathcal{Q}_5$ also has at least eight path components corresponding to the unknot. It is not clear that these components are contractible or not. 


\section{Conclusion}\label{sec5}

We have seen that the topology on a set of polynomial knots of degree at most $d$ depends on the coefficients of the component polynomials of knots belonging to that set. If the set of polynomial knots is flexible in the sense that many coefficients of the component polynomials can be zero then the space is path connected. In such a space a polynomial knot that is topologically representing a nontrivial knot can also be joined to a trivial knot by a path in that space. This is possible because the polynomial knots are the long knots (the ends are open). When we have a polynomial knot in such space, by a one parameter family of knots, our knot can slowly shift towards the end and gets opened up to become a trivial knot. For example, a polynomial knot $\phi:\ro\to\rt$ given by 
\begin{align*}
t\mapsto&\hf (51.84, -50.276, 0) + (-164.016, 160.508, -35.843)\hp t + (-31.92, 32.439, 187.195)\hp t^2\\
&\hp + (8.5, -29.11, 11.283)\hp t^3 + (1, -1.5, -19.116)\hp t^4 + (0, 1, -0.48)\hp t^5 + (0, 0, 0.5)\hp t^6
\end{align*} 
\noindent represents the figure-eight knot (see Figure \ref{fig2.1}). This knot is in $\mathcal{Q}_6$ and by Theorem \ref{th13}, it cannot be joined to a trivial knot by a path in $\mathcal{Q}_6$. However $\phi$ is an element of the space $\mathcal{O}_6$ also. In this space, we have a path $F:\zo\to\mathcal{O}_6$ given by $s\mapsto F_s$, where
\begin{align*}
F_s(t)&= (0, 0, 0.5)\hp(1 - s)^5 t^6 + (0, 1, -0.48)\hp(1 - s)^4 t^5 + (1, -1.5, -19.116)\hp(1 - s)^3 t^4\\ 
&\hskip4mm+ (8.5, -29.11, 11.283)\hp(1 - s)^2 t^3 + (-31.92, 32.439, 187.195)\hp(1 - s) t^2\\
&\hskip4mm+ (-164.016, 160.508, -35.843)\hp t + (51.84, -50.276, 0)\hp(1 - s)\hp.
\end{align*}
This path joins the polynomial knot $\phi$ to an unknot $\psi:\ro\to\rt$ given by $t\mapsto(-164.016\hp t, 160.508\hp t, -35.843\hp t)$. A depiction of this path is given in Figure \ref{fig2}. A similar thing happens in the space $\mathcal{P}_6$ also. That is, any non-trivial knot in $\mathcal{P}_6$ can be joined to a trivial knot by a path in $\mathcal{P}_6$.
\begin{figure}[H]
	\centering
	\begin{subfigure}{0.48\textwidth}
		\centering
		\includegraphics[scale=0.35]{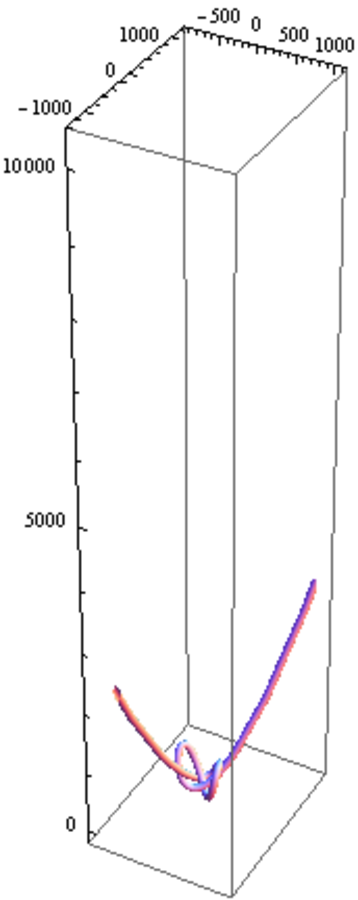}
		\caption{$F_s$ for $s=0$}
		\label{fig2.1}
	\end{subfigure}\quad
	\begin{subfigure}{0.48\textwidth}
		\centering
		\includegraphics[scale=0.34]{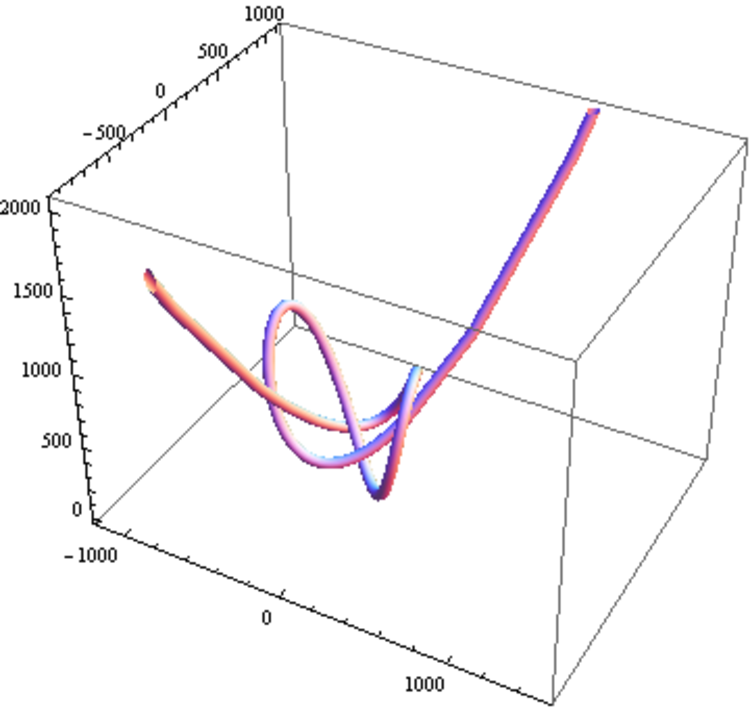}
		\caption{$F_s$ for $s=1/3$}
		\label{fig2.2}
	\end{subfigure}\\[10pt]
	\begin{subfigure}{0.48\textwidth}
		\centering
		\includegraphics[scale=0.3]{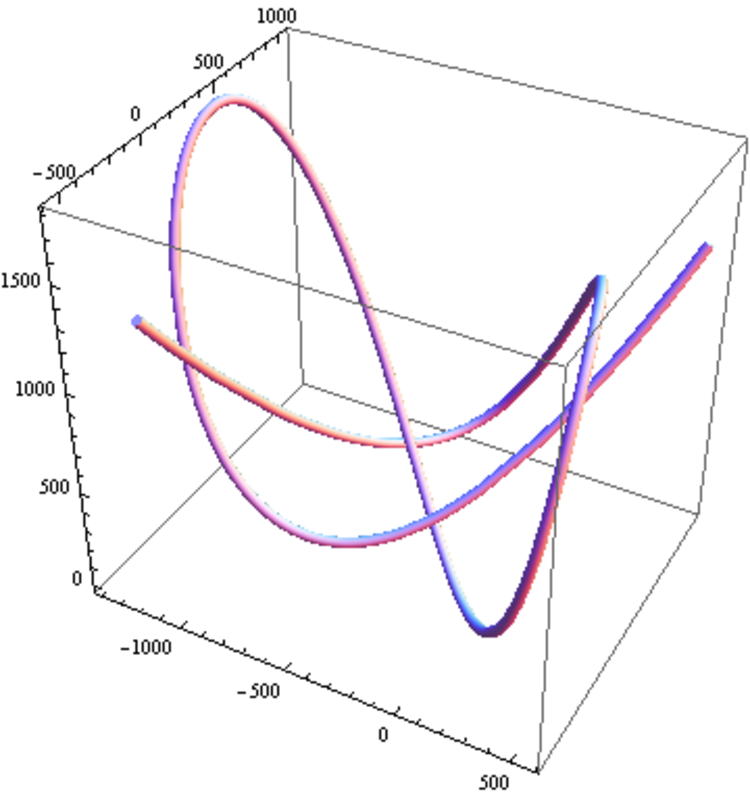}
		\caption{$F_s$ for $s=2/3$}
		\label{fig2.3}
	\end{subfigure}\quad
	\begin{subfigure}{0.48\textwidth}
		\centering
		\includegraphics[scale=0.41]{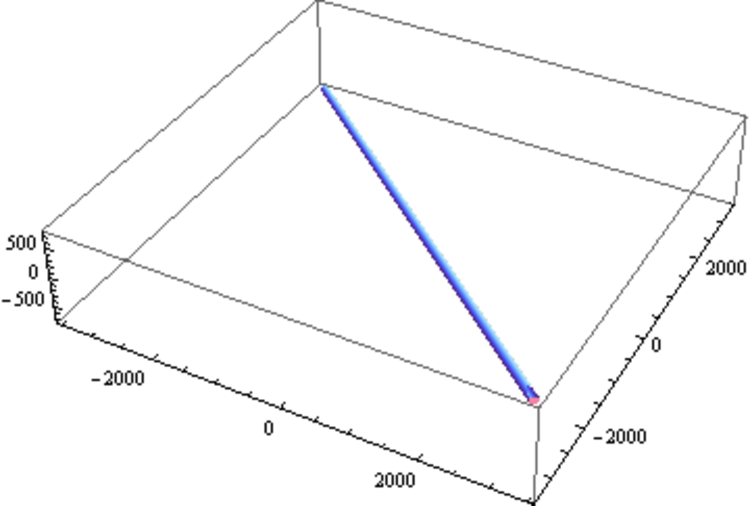}
		\caption{$F_s$ for $s=1$}
		\label{fig2.4}
	\end{subfigure}
	\caption{Depiction of the path $F:\zo\to\od$ connecting a polynomial figure-eight knot $\phi$ to a trivial polynomial knot $\psi$.}
	\label{fig2}
\end{figure}

We have seen that the space $\pd$, for $d\geq3$, is path connected. However, its homotopy type is not known. It will be interesting to explore the homotopy type of this space. The exact number of path components are known for the spaces $\mathcal{Q}_2$, $\mathcal{Q}_3$ and $\mathcal{Q}_4$ and we have proved that they are contractible. Also, we have shown that the path components of the space $\mathcal{Q}_5$ corresponding to the trefoil knot and its mirror image are contractible. Thus, in general, we may guess that the path components of the space $\qd$, for $d\geq5$, are contractible. Note that the exact number of path components of the space $\qd$, for $d\geq5$, are still not clear.

\section*{Acknowledgment}\label{sec6}

The first author is thankful to the University Grants Commission, India for the support of Research Fellowship for his PhD work. 


\end{document}